\documentclass[11pt]{amsart}
\usepackage[foot]{amsaddr}
\usepackage{ifxetex}
\ifxetex
  \usepackage[no-math]{fontspec}
\else
\fi

\usepackage{amsmath}
\usepackage{amssymb}
\usepackage{amsthm}
\usepackage{fullpage}
\usepackage[T1]{fontenc} % best for Western European languages
\usepackage{textcomp} % required to get special symbols
\usepackage{eulerpx, eucal, eufrak, tgpagella}
\usepackage{multicol}
\usepackage{listings}
\usepackage{hyperref, color}
\hypersetup{colorlinks=true,citecolor=blue, linkcolor=blue, urlcolor=blue}
\usepackage[linesnumbered,boxed,ruled,vlined]{algorithm2e}
\usepackage{bm}
\usepackage{bbm}
\usepackage[numbers]{natbib}
\usepackage{xcolor}
\usepackage{enumerate} 
\usepackage{enumitem}
\usepackage{tabularx}
\usepackage{array}
\usepackage{cleveref}
\usepackage{mathrsfs}
\usepackage{tikz}
\usetikzlibrary{arrows.meta}
\usepackage{float}
\newcolumntype{L}[1]{>{\raggedright\arraybackslash}p{#1}}
\newcolumntype{C}[1]{>{\centering\arraybackslash}m{#1}}
\newcolumntype{R}[1]{>{\raggedleft\arraybackslash}p{#1}}
  
\usepackage{makecell}

\usepackage{footnote}
\makesavenoteenv{tabular}

\newcommand{\DTV}[2]{d_{\mathrm{TV}}\left({#1},{#2}\right)}

\newcommand{\e}{\mathrm{e}}

\renewcommand{\epsilon}{\varepsilon}

\newtheorem{theorem}{Theorem}[section]

\newtheorem*{claim*}{Claim}

\newtheorem{fact}[theorem]{Fact}
\newtheorem{lemma}[theorem]{Lemma}
\newtheorem{proposition}[theorem]{Proposition}

\theoremstyle{definition}

\newtheorem{definition}[theorem]{Definition}
\newtheorem{remark}[theorem]{Remark}
\newtheorem*{remark*}{Remark}

\renewcommand{\emptyset}{\varnothing}

\newcommand{\norm}[1]{\left\Vert#1\right\Vert}
\newcommand{\set}[1]{\left\{#1\right\}}
\newcommand{\tuple}[1]{\left(#1\right)} 

\newcommand{\tp}{\tuple}

\renewcommand{\d}{\,\-d}

\newcommand{\abs}[1]{\left\vert#1\right\vert}

\newcommand{\ci}{coupling independent}

\def\*#1{\boldsymbol{#1}} % Use \*A for \mathbf{A}
\def\+#1{\mathcal{#1}} % Use \+A for \mathcal{A}
\def\-#1{\mathrm{#1}} % Use \-A for \mathrm{A}
\def\^#1{\mathscr{#1}} % Use \^A for \mathscr{A}

\usepackage{todonotes}

\usepackage{xifthen}

\def\oPr{\*{Pr}}
\renewcommand{\Pr}[2][]{ \ifthenelse{\isempty{#1}}
  {\oPr\left[#2\right]}
  {\oPr_{#1}\left[#2\right]} } % Use \Pr[a]{b} for \mathbf{Pr}_a[b], \Pr{b} for  \mathbf{Pr}[b]

\def\oE{\mathbb{E}}
\newcommand{\E}[2][]{ \ifthenelse{\isempty{#1}}
  {\oE\left[#2\right]}
  {\oE_{#1}\left[#2\right]} }

\def\oVar{\*{Var}}
\newcommand{\Var}[2][]{ \ifthenelse{\isempty{#1}}
  {\oVar\left[#2\right]}
  {\oVar_{#1}\left[#2\right]} }

\def\oEnt{\mathbf{Ent}}
\newcommand{\Ent}[2][]{ \ifthenelse{\isempty{#1}}
  {\oEnt\left[#2\right]}
  {\oEnt_{#1}\left[#2\right]} }

\newcommand{\relaxT}[2][]{
  \ifthenelse{\isempty{#2}}
  {t_{\mathrm{rel}}^{\mathrm{#1}}}
  {t_{\mathrm{rel}}^{\mathrm{#1}}(#2)}
}

\def\0{-1}
\def\1{+1}

\usepackage{multirow}
\usepackage{tikz}

\def\as#1#2#3{#1^{#2 \gets #3}}

\title{A near-linear time sampler for the Ising model with external field}
\date{}
\author{Xiaoyu Chen}
\author{Xinyuan Zhang}
\address[Xiaoyu Chen, Xinyuan Zhang]{State Key Laboratory for Novel Software Technology, Nanjing University, 163 Xianlin Avenue, Nanjing, Jiangsu Province, China. \textnormal{E-mails: \url{chenxiaoyu233@smail.nju.edu.cn}, \url{zhangxy@smail.nju.edu.cn}}}

\begin{document}
\begin{abstract}
  %强调一下对称情形 (symmetrically, lambda>1).
  %probably substitute with (on general graphs with bounded or unbounded maximum degree.)
  We give a near-linear time sampler for the Gibbs distribution of the ferromagnetic Ising models with edge activities $\*\beta > 1$ and external fields $\*\lambda<1$ (or symmetrically, $\*\lambda>1$) on general graphs with bounded or unbounded maximum degree.
  
  Our algorithm is based on the field dynamics given in~\cite{chen2021rapid}. 
  We prove the correctness and efficiency of our algorithm by establishing spectral independence of distribution of the random cluster model and the rapid mixing of Glauber dynamics on the random cluster model in a low-temperature regime, which may be of independent interest.
\end{abstract}
\maketitle
\section{Introduction}
% \todo{introduce others work without bias.}
% \todo{[...] should not be the start of a sentence}
%\todo{introduce RC model's history}
The Ising model~\cite{ising1925beitrag} introduced by Ising and Lenz is an extensively studied statistical physics model which leads to many inspiring discoveries in physics, discrete probability, machine learning, and theoretical computer science.
Let $G = (V, E)$ be an undirected graph with $n$ vertices and $m$ edges, $\*\beta \in (1,+\infty)^E$ be the \emph{edge activities}, and $\*\lambda \in [0, 1]^V$ be the \emph{external fields}.
The \emph{Gibbs distribution} $\mu^{\-{Ising}}_{\*\beta, \*\lambda}$ over $2^V$ of the \emph{ferromagnetic Ising model} is defined by
\begin{align*}
  \forall S \subseteq V, \quad \mu^{\-{Ising}}_{\*\beta, \*\lambda}(S) := \frac{1}{Z^{\-{Ising}}_{\*\beta, \*\lambda}} \prod_{e \in m(S)} \beta_e \prod_{v \in S} \lambda_v,
\end{align*}
where $m(S) := \{e \in E \mid e\cap S = e \text{ or } e\cap S = \emptyset\}$ denotes the set of ``monochromatic'' edges, and $Z^{\-{Ising}}_{\*\beta, \*\lambda} := \sum_{S\subseteq V} \prod_{e \in m(S)} \beta_e \prod_{v \in S} \lambda_v$ is known as the partition function.
A major problem is to sample from the Gibbs distribution of the ferromagnetic Ising model.% $\mu^{\-{Ising}}_{\*\beta, \*\lambda}$.

% The first algorithm for this task that people may come up to their mind
One of the most well-known approaches is the Markov Chain Monte Carlo (MCMC) method. 
The Glauber dynamics, also known as the Gibbs sampler, is an example of this method.
% Let $\Delta$ be the maximum degree of a graph $G$ and $\beta_c(\Delta) := \Delta/(\Delta - 2)$ be the critical temperature.
There are numerous researches establishing the rapid mixing results of Glauber dynamics~\cite{mossel2013exact, chen2020rapid, chen2020optimal, chen2021rapid, anari2021entropic, anari2021entropicII, chen2022localization} when $\*\beta < \beta_c(\Delta)$, where
$\Delta$ is the maximum degree of a graph $G$ and $\beta_c(\Delta) := \Delta/(\Delta - 2)$ is the critical threshold.
However, when $\*\beta > \beta_c(\Delta)$, it is known that there exist graphs such that the Glauber dynamics is exponentially slow in the size of the graph~\cite{GerschenfeldM07}.

Even though the Glauber dynamics fails to be efficient, there still exist fast algorithms to sample from the Gibbs distribution of the ferromagnetic Ising model.
The random cluster model~\cite{fortuin1972random-I, fortuin1972random-II, fortuin1972random-III} and subgraph-world model are two statistical mechanics models that are closely related to the Ising model.
%The random cluster model was first introduced in the statistical mechanics literature by \cite{}.
Leveraging the connection between the partition function of the Ising model and the subgraph-world model~\cite{newell1953theory},~\cite{jerrum1993polynomial} showed that the $1/2$-lazy Metropolis chain on the subgraph-world model converges rapidly by using the technique of canonical path~\cite{jerrum1989approximating}.
%In a seminal work of Jerrum and Sinclair \cite{jerrum1993polynomial}, they established the connection between the partition function of the Ising model and the subgraph-world model~\cite{newell1953theory} and showed that the $1/2$-lazy Metropolis chain on the subgraph-world model converges rapidly by using the technique of canonical path~\cite{jerrum1989approximating}.
%This indeed implies an algorithm for Ising model for the whole $\*\beta > 1$ regime.
% After that, a coupling between random cluster model and the subgraph-world model is noticed by \cite{grimmett2009random}.
% This coupling was used to lift the canonical path of Glauber dynamics on subgraph-world model to a multicommodity flow, which is a generalization of canonical path introduced by \cite{sinclair1992improved}, of Glauber dynamics on random cluster model by later works~\cite{guo2018random, feng2022sampling}, which implies that the Glauber dynamics on random cluster model mixes rapidly.
The following works~\cite{guo2018random,feng2022sampling} established a similar mixing time of Glauber dynamics on the random cluster model via multicommodity flow based on the canonical path in~\cite{jerrum1993polynomial} and the coupling in~\cite{grimmett2009random}.
All these results can be translated into fast Ising samplers that run in time $\widetilde{O}_{\*\beta, \*\lambda}(m^3)$ when $\*\beta > 1$ and $\*\lambda < 1$.
% We note that it holds for general graph even without bounded degree assumption.
Furthermore, these samplers also work when $\*\lambda = 1$, where the running time degenerates to $\widetilde{O}_{\*\beta}(n^4m^3)$.
It is also worth mentioning that for the case $\*\lambda = 1$, there is a specific Markov chain called worm process on the Prokof'ev-Svistunov measure of the Ising model studied in statistical physics, which is proved to be rapid mixing~\cite{collevecchio2016worm}.

With bounded degree assumption, \cite{chen2021spectral} 
% used Asano-Ruelle contractions~\cite{asano1970theorems, ruelle1971extension} to get a zerofreeness region of the subgraph-world model and then translated this region to a bound on the recently developed notion of spectral independence~\cite{anari2020spectral, chen2020optimal}.
% Using this method, they 
proved the optimal mixing time of Glauber dynamics on the subgraph-world model via spectral independence, implying a fast Ising sampler that runs in $\Delta^{O_{\*\beta, \*\lambda}(\Delta)} \cdot \widetilde{O}(m)$ when $\*\beta > 1$ and $\*\lambda < 1$.

Apart from MCMC based method, algorithms based on zero-freeness property and polymer model also achieve polynomial running time with restrictions.
By the celebrated Lee-Yang circle theorem~\cite{lee1952statistical} and the polynomial interpolation algorithm framework~\cite{barvinok2016combinatorics, patel2017deterministic, peters2019conjecture}, there is an algorithm for sampling the ferromagnetic Ising model that runs in time $n^{O_{\*\lambda}(\log\Delta)}$ for $\*\beta > 1$ and $\*\lambda < 1$~\cite{liu2019ising}, where $\Delta$ is the maximum degree of the graph.
% \cite{chen2021spectral} used Asano-Ruelle contractions~\cite{asano1970theorems, ruelle1971extension} to get a zerofreeness region of the subgraph-world model and then translated this region to a bound on the recently developed notion of spectral independence~\cite{anari2020spectral, chen2020optimal}.
% Using this method, they showed an optimal mixing time of Glauber dynamics on subgraph-world, which implies a fast Ising sampler that runs in $\Delta^{O_{\*\lambda}(\Delta)} \cdot \widetilde{O}_{\*\beta, \*\lambda}(m)$ when $\*\beta > 1$ and $\*\lambda < 1$.
% The polymer representation of Ising model has also been considered in order to design fast samplers.
While algorithms based on the polymer model usually require graph $G$ to be an $\alpha$-expander for some constant $\alpha > 0$ and parameter $\*\beta = \Omega(\log(\Delta)/\alpha)$.
An algorithm of this type with running time $n^{O_{\*\beta, \alpha}(\log \Delta)}$ was given in~\cite{jenssen2020algorithms}, and
follow-up works~\cite{chen2021fast, blanca2022fast} improved the running time to $\widetilde{O}_{\*\beta, \alpha}(n)$.
% We remark that in this series of works, the result actually holds for a multicolor generalization of Ising model called \emph{Potts model}.
% However, they only considered the no field ($\*\lambda = 1$) case.

Besides, there are many other fast samplers for the Ising model on special families of graphs \cite{martinelli2003ising, galanis2019swendsen, blanca2021swendsen, blanca2022entropy, gheissari2022low}, such as lattice graph and torus graph.

In summary, no algorithms in previous studies run faster than cubic time without any assumption on graphs or parameters.
% if one wants to handle the general instances, it has to tolerate cubic running time; if one wants to be fast, then it is restricted to special familes of instances.
% Ideally, people want the best of both worlds, and i
It is natural to ask the following question:

\begin{center}
  \textit{Are there faster algorithms to sample from Gibbs distribution of the Ising model in the general case?}
\end{center}

In this paper, we answer this question in the affirmative.
%\todo{verify $\lambda = 0$}
\begin{theorem} \label{thm:main}
  Let $\delta_{\beta}, \delta_{\lambda} \in (0,1)$ be constants, and $\mu$ be the Gibbs distribution of the ferromagnetic Ising model specified by graph $G = (V, E)$, parameters $\*\beta \in [1+\delta_{\beta},+\infty)^E$ and $\*\lambda \in [0, 1-\delta_{\lambda}]^V$.
  % Let $\delta_\beta = \min_{e \in E} (\beta_e - 1)$ and $\delta_\lambda = \min_{v \in V} (1 - \lambda_v)$ be two positive constants.
  There exists an algorithm that samples $X$ satisfying $\DTV{X}{\mu} \leq \epsilon$ for any given parameter $\epsilon > 0$ within running time
  % \[O_{\*\beta, \*\lambda}(m \cdot \-{polylog}(n) \cdot \-{poly}(1/\epsilon)),\]
  % \[m(\log n \cdot \epsilon^{-2} \cdot (\delta_\beta^{-1} + 1))^{O(\delta_\lambda^{-8})},\]
  \begin{align*}
    m \cdot (\log n)^{O_{\delta_{\beta},\delta_{\lambda}}(1)},
  \end{align*}
  where $m$ is the number of edges and $n$ is the number of vertices.
\end{theorem}

\begin{remark}
  Let $\mu$ be the Gibbs distribution of the Ising model with parameters $\*\beta \in \mathbb{R}_{>1}^E$ and $\*\lambda \in \mathbb{R}_{>1}^V$, and define $\overline{\mu}$ by letting $\overline{\mu}(S) = \mu(V \setminus S)$ for each $S \subseteq V$.
  Note that $\overline{\mu}$ is the Gibbs distribution of the Ising model with parameters $\*\beta$ and $\*\lambda^{-1}  \in (0,1)^V$.
  Therefore, we can sample from $\overline{\mu}$ via the sampler in \Cref{thm:main}, which implies a sampler for $\mu$. 
\end{remark}
%\todo{substitute with$ m log^{O(1/1 - \lambda)} 1/\epsilon^2$.. something, or gap version.}

Compared to previous works, our algorithm could handle general instances while it only takes a near-linear running time when parameters are bounded away from $\*1$.
We give a detailed comparison between \Cref{thm:main} and previous results in \Cref{tab:comparison}.

\begin{table}[h!]
  \centering
  \begin{tabular}{|c|c|c|}
    \hline
    results & running time & requirements \\
    \hline
    %\multirow{2}{*}{\cite{jerrum1993polynomial, guo2018random, feng2022sampling}}
    \multirow{2}{*}{\cite{jerrum1993polynomial}}
    & $\widetilde{O}_{\*\beta, \*\lambda}(m^3)$ & $\*\lambda < 1$ \\
    \cline{2-3}
    & $\widetilde{O}_{\*\beta}(n^4m^3)$ & $\*\lambda = 1$ \\
    \hline
    %\cite{lee1952statistical, barvinok2016combinatorics, patel2017deterministic, guo2018uniqueness, peters2019conjecture}
    \cite{liu2019ising} & $n^{O_{\*\lambda}(\log \Delta)}$ & $\*\lambda < 1$ \\
    \cline{1-2}
    %\cite{anari2020spectral, chen2020optimal, chen2021spectral}
    \cite{chen2021spectral} & $\Delta^{O_{\*\beta, \*\lambda}(\Delta)} \cdot \widetilde{O}(m)$ & bounded degree \\
    \hline
    %\multirow{2}{*}{\cite{jenssen2020algorithms, chen2021fast, blanca2022fast}}
    \multirow{2}{*}{\cite{chen2021fast}}
     & \multirow{2}{*}{$\widetilde{O}_{\*\beta, \alpha}(n)$} & $\alpha$-expander, $\*\lambda = 1$ \\
     & & $\*\beta = \Omega(\log(\Delta)/\alpha)$ \\
    \hline
    This work (\Cref{thm:main}) & $\widetilde{O}_{\*\beta,\*\lambda}(m)$ & $\*\lambda < 1$ \\
    \hline
  \end{tabular}
  \caption{Comparison of running times and requirements for Ising samplers}
  \label{tab:comparison}
\end{table}
%\todo{table: precise time complexity.}% 把各个结果分开来，每个工作对应相应的bound。另外可以删掉polymer的情形，是一个特殊图上的算法。}

%\todo{citation of Edwards-Sokal coupling}
As in \cite{guo2018random, feng2022sampling}, we leverage the Edwards-Sokal coupling~\cite{edwards1988gen} (see~\Cref{prop:ES-coupling}), which reduces the task of sampling from the distribution of the ferromagnetic Ising model to the random cluster model (see~\Cref{sec:def-RC} for formal definition).
\begin{theorem} \label{thm:alg-RC}
  Let $\delta_{p}, \delta_{\lambda} \in (0,1)$ be constants and $\mu$ be the distribution of a random cluster model specified by graph $G=(V,E)$, parameters $\*p \in [\delta_p,1)^E$ and $\*\lambda \in [0,1-\delta_{\lambda}]^V$.
  % Let $p_{\min} = \min_{e\in E} p_e$ and $\delta_\lambda = \min_{v \in V} (1 - \lambda_v)$ be two positive constants.
  There is an algorithm that samples $X$ satisfying $d_{\-{TV}}(X,\mu) \leq \epsilon$ for any given parameter $\epsilon>0$ within running time
  %\[O_{\*\lambda, \*p}(m \cdot \-{polylog}(n) \cdot \-{poly}(1/\epsilon)),\]
  % \[m(\log n \cdot \epsilon^{-2} \cdot p^{-1}_{\min})^{O(\delta_\lambda^{-8})},\]
  \[
    m \cdot \tp{\log n}^{O_{\delta_{p},\delta_{\lambda}}(1)},
  \]
  where $m$ is the number of edges and $n$ is the number of vertices.
\end{theorem}
The proof of \Cref{thm:main} assuming \Cref{thm:alg-RC} is deferred to \Cref{sec:main-proof}.

In the recent progress on high-dimensional expansion and the analysis of Glauber dynamics, a new Markov chain called \emph{field dynamics} has played an important role~\cite{chen2021rapid, anari2021entropicII, chen2022optimal, chen2022localization}.
The field dynamics was originally used to obtain a boosted optimal spectral gap or modified log-Sobolev constant from a sub-critical regime. 
% all the way up to the critical threshold.
However, it turns out that the field dynamics could also be used to design fast sampler.
A recent work \cite{anari2021entropicII} used field dynamics with interleaved systematic scans to build a fast sampler for the hardcore model.
In our result, \Cref{thm:alg-RC} is another example in which field dynamics is used as an algorithmic tool to design fast sampler.
The proof of \Cref{thm:alg-RC} is outlined in \Cref{sec:proof-outline}.

The key ingredients in \Cref{thm:alg-RC} are to establish
spectral independence of the random cluster model for graphs with both bounded and unbounded
maximum degree, and to prove the mixing results for the random cluster model in a low-temperature
regime.
The previous work~\cite{chen2021spectral} established an $O_{\Delta, \*p, \*\eta}(1)$ bound for the spectral independence of the subgraph-world model (see \Cref{sec:def-SG}) by using the analytical property offered by the zero-freeness region, which leads to a $\Delta^{O_{\*\beta, \*\lambda}(\Delta)}$ factor in the running time of the sampler for the Ising model or the random cluster model.
In this work, we remove the dependency on $\Delta$ and prove an $O_{\*\eta}(1)$ spectral independence for the subgraph-world model by using a novel coupling based argument (see \cref{lm:sg-coupling}).
Unlike the previous analysis in~\cite{chen2021spectral}, this coupling based argument enables us to lift the spectral independence bound from the subgraph-world model to the random cluster model by using the nature coupling between these models.
Therefore, we are able to prove the first $O_{\*\lambda}(1)$ bound for the spectral independence of the random cluster model (see \cref{lem:inf-RC}).
Finally, in \cref{sec:good-regime}, we use the coupling with stationary argument to show that the Glauber dynamics on the random cluster model mixes rapidly in a low-temperature regime (i.e. when $p_{\min}$ is close to $1$).
With these ingredients, we develop a near-linear time sampler via the field dynamics for the random cluster model on general graphs with both bounded
and unbounded maximum degree.

We remark that the recently updated version~\cite{feng2022samplingb} of~\cite{feng2022sampling} proved the optimal mixing time of
the Glauber dynamics of the random cluster model when the fields are bounded away from $\*1$ and the
maximum degree of graphs is bounded by a universal constant. Their proof is also based on the
high-dimensional expander.

\subsection{Open problems}
In this paper, we developed a near-linear time sampler for Ising models with parameters $\*\beta > 1$ and $\*\lambda<1$ (or symmetrically, $\*\lambda > 1$). 
It still leaves several open problems.
\begin{itemize}
  \item Develop a good sampler for the ferromagnetic Ising model when field $\*\lambda = \*1$. 
        Our algorithm fails due to the exponential reliance on the gap of $\*\lambda$ and $\*1$, which stems from the analysis based on high dimensional expander technique. 
        Therefore, It is still tempting to surpass Jerrum and Sinclair's algorithm~\cite{jerrum1993polynomial} in this case.
  \item Better analysis on the Glauber dynamics and the Swendsen-Wang dynamics of the random cluster model. 
        These simple yet powerful dynamics are of great interests in the study of random cluster model \cite{gore1999swendsen,guo2018random,galanis2019swendsen,blanca2021swendsen,feng2022sampling}.
        Though, the current mixing time bounds for these dynamics on general graphs are still far from optimal. 
        We hope our techniques and results could be an inspiration for works in this field. 
\end{itemize}

% However, our result crucially relies on the condition where $\*\lambda < 1$.
% This because in the our current analysis, we need a constant total influence bound for random cluster model, which is currently achieved by showing that the total influence in the subgraph-world model is bounded by a constant (see \Cref{sec:def-SG} for the definition of subgraph-world model).
% However, when $\*\lambda = 1$, this bound for the subgraph-world model may become order $n$ for some special family of graphs like cycles of length $n$.
% A natural question is if we can analysis the total influence for random cluster model directly so that we could bypass this issue.

\section{Preliminaries}

\subsection{Notation}
Let $\mu$ be a distribution over $2^U$ for some ground set $U$, and $\tau,\Lambda$ be subsets of $U$.
$\+P_{\tau,\Lambda}$ denote the set of configurations $S \subseteq U$ that agree with $\tau$ on $\Lambda$, i.e.
\begin{align*}
  \+P_{\tau,\Lambda} = \{S \subseteq U \mid S \cap \Lambda = \tau \cap \Lambda\}.
\end{align*}
The distribution conditional on event $\+E \subseteq 2^U$ is defined by\footnote{Event $\+E$ must satisfy $\mu(\+E) = \sum_{S \in \+E} \mu(S) > 0$.}
\begin{align*}
  \forall S \subseteq U, \quad \mu(S \mid \+E) =
  \begin{cases}
    \frac{\mu(S)}{\mu(\+E)} & \text{if }S \in \+E,\\
    0& \text{otherwise.}
  \end{cases}
\end{align*}
For simplicity, we denote by $\mu\tp{\cdot \mid i}$ (resp. $\mu(\cdot \mid \overline{i})$) the distribution $\mu\tp{\cdot \mid \+P_{\{i\}, \{i\}}}$ (resp. $\mu\tp{\cdot \mid \+P_{\emptyset, \{i\}}}$) projected on $U \setminus \{i\}$ for some $i \in U$.
Furthermore, we denote by $\mu(i)$ (resp. $\mu(\overline{i})$) be the probability $\Pr[S \sim \mu]{i \in S}$ (resp. $\Pr[S \sim \mu]{i \not \in S}$).

Let $\*p \in [0,1]^U$. The distribution $\mu = \bigotimes_{i \in U} \mathrm{Ber}(p_i)$ is defined by
\begin{align*}
  \forall S \subseteq E, \quad \mu(S) = \prod_{i \in S} p_i \prod_{i \not\in S} (1-p_i).
\end{align*}

We will write $\*1$ be the constant vector with value $1$, and $\*1_u = ({\*1}_{i = u})_{i \in U}$ for some $u \in U$.
Lastly, let $X, Y \subseteq U$ be two subsets of the ground set $U$, we use $X \oplus Y := (X\setminus Y) \cup (Y\setminus X)$ to denote the symmetric difference between $X$ and $Y$.
\subsection{Markov chains, entropy and mixing time}

\subsubsection{Basic definitions}
Let $(X_t)_{t \in \mathbb{N}}$ be a Markov chain over a finite state space $\Omega$ with transition matrix $P=(p_{x,y})_{x,y \in \Omega} \in \mathbb{R}_{\ge 0}^{\Omega \times \Omega}$. 
$(X_t)_{t \in \mathbb{N}}$ is \emph{irreducible}, if for any $x,y \in \Omega$, there exists $t > 0$, such that $P^{t}(x,y) > 0$.
$(X_t)_{t \in \mathbb{N}}$ is \emph{aperiodic}, if for any $x \in \Omega$, $\gcd \left\{ t \in \mathbb{N}_{>0} \mid P^{t}(x,x)>0 \right\} = 1$. 
A distribution $\mu$ over $\Omega$ is a \emph{stationary distribution} of $(X_t)_{t \in \mathbb{N}}$, if $\mu=\mu P$.
The fundamental theorem of Markov chain says that a Markov chain $(X_t)_{t \in \mathbb{N}}$ has a unique stationary distribution, if the Markov chain is irreducible and aperiodic.
A distribution $\mu$ over $\Omega$ is \emph{reversible} with respect to $(X_t)_{t \in \mathbb{N}}$, if $\mu$ satisfies the \emph{detailed balance condition}, i.e. $\mu(x)P(x,y)=\mu(y)P(y,x)$ holds for all $x,y \in \Omega$. 
It is known that $\mu$ is the stationary distribution of $(X_t)_{t \in \mathbb{N}}$ if $\mu$ is reversible with respect to the Markov chain.

Let $\mu, \nu$ be two distributions over the finite state space $\Omega$, the \emph{total variation distance} is defined as
\begin{align*}
  d_{\mathrm{TV}}(\mu,\nu) = \max_{S \subseteq \Omega} \abs{\mu(S) - \nu(S)} = \frac{1}{2} \sum_{\sigma \in \Omega} \abs{\nu(\sigma)-\mu(\sigma)}. 
\end{align*}
Suppose $\mu$ is the stationary distribution of a Markov chain $(X_t)_{t \in \mathbb{N}}$ with transition matrix $P$. 
The \emph{mixing time} initialized from configuration $X_0$ is defined by
\begin{align*}
  T_{\mathrm{mix}}(\epsilon, X_0) = \min \{t \in \mathbb{N} \mid d_{\mathrm{TV}}(P^t(X_0,\cdot), \mu) < \epsilon\}.
\end{align*}
The \emph{mixing time} is defined by $T_{\mathrm{mix}}(\epsilon) = \max_{X_0 \in \Omega} T_{\mathrm{mix}}(\epsilon,X_0)$.
\subsubsection{Glauber dynamics}
Let $\mu$ be a distribution over $2^U$ on a finite ground set $U$. 
One of the most famous single-site dynamics is the \emph{Glauber dynamics} (a.k.a. Gibbs sampler).
In each step, the Glauber dynamics updates configuration $X \in 2^U$ according to the following rules:
\begin{itemize}
  \item pick $i \in U$ uniformly at random;
  \item update $X$ according to distribution $\mu\tp{\cdot \mid \+P_{X, U\setminus\{i\}}}$.
\end{itemize}
It can be verified that $\mu$ is reversible with respect to Glauber dynamics.

%An adaptive select-update dynamics, field dynamics, was introduced in \cite{chen2021rapid}.
%In each step, the field dynamics with parameter $\theta \in (0,1)$ updates configuration $X \in 2^U$ according to the following rules:
%\begin{itemize}
%  \item generate $S \sim \bigotimes_{i \in U} \mathrm{Ber}\tp{1 - \theta}$ and let $R = S \cap X$;
%  \item update the status of $S \setminus R$ in $X$ according to distribution $(\theta * \mu) \tp{\cdot \mid \+P_{X, R}}$.
%\end{itemize}
%It has been verified in \cite{chen2021rapid} that $\mu$ is reversible with respect to field dynamics with any parameter $\theta \in (0,1)$.

\subsubsection{Entropy decay and mixing time}
The relation between the functional inequalities and the mixing time of the Markov chain has been studied in literature~\cite{diaconis1996log, bobkov2006modified}. We now introduce the decay of the relative entropy and its implication on the mixing time of the Markov chain.

Let $\mu,\nu$ be distributions over finite state set $\Omega$ and $\nu$ is absolutely continuous with respect to $\mu$. The \emph{relative entropy} (known as \emph{Kullback-Leibler divergence}) between $\nu$ and $\mu$ is defined as 
\begin{align*}
  D_{\mathrm{KL}}(\nu \parallel \mu) = \sum_{\sigma \in \Omega} \nu(\sigma) \log \tp{\frac{\nu(\sigma)}{\mu(\sigma)}},
\end{align*}
with convention $0\cdot \infty = 0$.

Let $(X_t)_{t \in \mathbb{N}}$ be a Markov chain with transition matrix $P$ and stationary distribution $\mu$. 
The relative entropy \emph{decays with rate $\alpha$}, if for any distribution $\nu$ absolutely continuous with respect to $\mu$,
\begin{align*}
  D_{\mathrm{KL}}(\nu P \parallel \mu P)\le (1-\alpha) D_{\mathrm{KL}}(\nu \parallel \mu).
\end{align*}
It is known (see, for example,~\cite[Lemma 2.4]{blanca2022entropy})
that the mixing time $T_{\mathrm{mix}}(\epsilon, X_0)$ satisfies
\begin{align}\label{eq:mixing-decay}
  T_{\mathrm{mix}}(\epsilon, X_0) \le \alpha^{-1}\tp{\log \log \frac{1}{\mu(X_0)} + \log \frac{1}{2\epsilon^2}}.
\end{align}

\subsection{Subgraph-world model, random cluster model, and relation of models}

% \subsubsection{Ising model}
% Let $G = (V, E)$ be an undirected graph, $\*\beta \in \mathbb{R}_{>1}^E$

\subsubsection{Generalized subgraph-world model} \label{sec:def-SG}
% \todo{choice of word}
Let $G=(V,E)$ be an undirected graph, $\*p\in [0,1]^E, \*\eta \in [0,1]^V$ and $\sigma \in \{0,1\}^V$. 
The weight of a configuration $S \subseteq E$ in the generalized subgraph-world model is defined by:
\begin{align*}
  w^{\-{GSW}}_{E,\*p, \*\eta, \sigma}(S) := \prod_{e \in S} p_e \prod_{f \in E \setminus S} (1-p_f) \prod_{\substack{v \in V\\ \abs{S \cap E_v} \equiv \sigma_v \tp{\text{mod } 2}}} \eta_v,
\end{align*}
where $E_v$ denotes the set of edges that is incident to $v$. 
For ease of notation, we may use $\*p^{S}, (1 - \*p)^{E\setminus S}$ to denote $\prod_{e\in S} p_e$ and $\prod_{f \in E\setminus S}(1 - p_f)$, respectively.
The distribution $\mu^{\mathrm{GSW}}_{E, \*p, \*\eta, \sigma}$ on $2^{E}$ is
\begin{align*}
  \forall S \subseteq E, \quad \mu^{\-{GSW}}_{E,\*p,\*\eta,\sigma}(S) := \frac{w^{\-{GSW}}_{E,\*p,\*\eta,\sigma}(S)}{Z^{\-{GSW}}_{E,\*p,\*\eta,\sigma}},
\end{align*}
where $Z^{\-{GSW}}_{E,\*p,\*\eta,\sigma} := \sum_{S \subseteq E} w^{\-{GSW}}_{E,\*p,\*\eta,\sigma}(S)$ is the partition function of this system
\footnote{$Z^{\mathrm{GSW}}_{E,\*p,\*\eta,\sigma}$ may equal to zero. In this case, the system is invalid and we will not consider such case.}.
%Occasionally, we may treat $S$ as a vector in $\{0,1\}^E$, and omit script when it is clear from context.
We remark that when $\*p' \in [0, 1]^{E'}$, where $E' \supseteq E$ is a superset of $E$, the distribution $\mu^{\-{GSW}}_{E, \*p', \*\eta, \sigma}$ is defined by 
\begin{align*}
  \mu^{\-{GSW}}_{E, \*p', \*\eta, \sigma} = \mu^{\-{GSW}}_{E, \*p'|_E, \*\eta, \sigma},
\end{align*}
where $\*p'|_E$ is the vector obtained by restricting $\*p'$ to $E$.
When $\sigma = \*1$, our definition matches the definition of the subgraph-world model~\cite{jerrum1993polynomial,feng2022sampling}. 
In this case, we may denote the distribution and partition function by $\mu^{\mathrm{SW}}_{E,\*p,\*\eta}$ and $Z^{\mathrm{SW}}_{E,\*p,\*\eta}$ instead.

\subsubsection{Random cluster model} \label{sec:def-RC}
Let $G=(V,E)$ be an undirected graph, $\*p\in [0,1]^E, \*\lambda \in [0,1]^V$ be parameters. 
The weight of a configuration $S \subseteq E$ in the random cluster model is defined by:
\begin{align*}
  w^{\-{RC}}_{E,\*p,\*\lambda}(S) := \prod_{e \in S} p_e \prod_{f \in E \setminus S} (1-p_e) \prod_{C \in \kappa(V, S)} \tp{1+\prod_{j \in C} \lambda_j}, 
\end{align*}
where we use $\kappa(V, S)$ to denote the set of connected components of graph $(V, S)$.
%where $C_1,C_2,\ldots,C_k$ be the connected components of graph $H=(V,S)$.
The distribution $\mu^{\-{RC}}_{E,\*p,\*\lambda}$ is defined by
\begin{align*}
  \forall S \subseteq E, \quad \mu^{\-{RC}}_{E,\*p,\*\lambda}(S) := \frac{w^{\-{RC}}_{E,\*p,\*\lambda}(S)}{Z^{\-{RC}}_{E,\*p,\*\lambda}},
\end{align*}
where $Z^{\-{RC}}_{E,\*p,\*\lambda} := \sum_{S \subseteq E} w^{\-{RC}}_{E,\*p,\*\lambda}(S)$ is the partition function of this system.
% \footnote{
When $\*\lambda = \*1$, our definition matches the classical definition of the random cluster model with $q=2$~\cite{fortuin1972random-I}.
% }
%Occasionally, we may treat $S$ as a vector in $\{0,1\}^E$, and omit script when it is clear from context.

\subsubsection{Relation between models}
A well-known holographic transformation~\cite{jerrum1993polynomial,valiant2008holographic} connects the partition functions of the ferromagnetic Ising model, the subgraph-world model and the random cluster model.
\begin{proposition}[\text{\cite[Lemma 2.1]{feng2022sampling}}] \label{prop:partition-func}
  Let $G=(V,E)$ be a graph, $\*\beta \in (1,+\infty)^E$ and $\*\lambda \in [0,1]^V$ be parameters, then
  \begin{align*}
    Z^{\mathrm{Ising}}_{\*\beta,\*\lambda} = \tp{\prod_{e \in E} \beta_e} Z^{\mathrm{RC}}_{E,\*p,\*\lambda} = \tp{\prod_{e \in E} \beta_e}\tp{\prod_{v \in V} (1+\lambda_v)} Z^{\mathrm{SW}}_{E,\*p/2,\*\eta},
  \end{align*} 
  where $\*p = 1-\*\beta^{-1} = \tp{1-\beta_e^{-1}}_{e \in E}$ and $\*\eta = \tp{\frac{1-\lambda_v}{1+\lambda_v}}_{v \in V}$.
\end{proposition}

%Furthermore, a simple coupling of Gibbs distributions of subgraph-world model and random cluster model has been established\cite{grimmett2009random,feng2022sampling}. 
%\begin{proposition}[Lemma 3.3, \cite{feng2022sampling}]
%  Let $G=(V,E)$ be an undirected graph, $\*p \in [0,1]^E$ and $\*\lambda \in [0,1]^V$ be parameters. 
%  The following procedure is a coupling of distributions $\mu^{\mathrm{SW}}_{E,\*p/2,\*\eta}$ and $\mu^{\mathrm{RC}}_{E,\*p,\*\lambda}$, where $\*\eta = \tp{\frac{1-\lambda_v}{1+\lambda_v}}_{v \in V}$.
%  \begin{itemize}
%    \item Sample $X \sim \mu^{\mathrm{SW}}_{E,\*p/2,\*\eta}$, $Z \sim \bigotimes_{f \in E} \mathrm{Ber}\tp{\frac{p_f}{2-p_f}}$ independently, and return $(X,X \cup Z)$.
%  \end{itemize} 
%\end{proposition}

The standard Edwards-Sokal coupling connects the Gibbs distribution of the ferromagnetic Ising model and the distribution of the random cluster model.
\begin{proposition}[\text{\cite[Proposition 2.3]{feng2022sampling}}] \label{prop:ES-coupling}
  Let $G = (V, E)$ be a graph, $\*\beta \in (1,+\infty)^E$ and $\*\lambda \in [0, 1]^V$ be parameters.
  Furthermore, let $\*p = 1 - \*\beta^{-1} = (1 - \beta_e)_{e\in E}$. 
  Construct $Y$ according to the following rules.
  \begin{itemize}
  \item initialize $Y = \emptyset$ and sample $X \sim \mu^{\-{RC}}_{E, \*p, \*\lambda}$;
  \item for each $C \in \kappa(V, X)$, add $C$ to $Y$ with probability $\frac{\prod_{i \in C} \lambda_i}{1 + \prod_{i \in C} \lambda_i}$.
  \end{itemize}
  Then, it holds that $Y \sim \mu^{\-{Ising}}_{\*\beta, \*\lambda}$.
\end{proposition}

\section{Proof outline} \label{sec:proof-outline}
In this section, we outline the proof of \Cref{thm:alg-RC}. 
We first summarize previous results of the field dynamics in \Cref{sec:field-dynamics}, and then introduce the field dynamics simulator, prove its correctness as an approximate sampler, and analyze its running time in \Cref{sec:simulation}.

\subsection{Field dynamics}\label{sec:field-dynamics}
%\todo{maybe introduce some background of FD here}
% Let $G = (V, E)$ be a graph with $n$ vertices and $m$ edges.
% Let $\*\lambda \in \mathbb{R}^n, \*p \in \mathbb{R}^m$ be two constant vectors satisfying $0 < \*\lambda, \*p < 1$.
% Let $\mu$ be the Gibbs distribution of the random cluster model specified by $G$, $\*p$, and $\*\lambda$, we define another distribution $\overline{\mu}$ by filpping the role of ``in'' and ``out'' for each edge, that is
The \emph{field dynamics} is an adaptive select-update dynamics, first introduced in~\cite{chen2021rapid}. 
  Intuitively, the field dynamics serves as a reduction from a critical instance to a sub-critical instance, i.e., instance that are much easier to handle.

Let $\mu$ be a distribution over $2^U$ on ground set $U$ and $\*\lambda \in \mathbb{R}_{> 0}^U$. 
Denote by $\*\lambda * \mu$ the distribution over $2^U$ satisfying
\begin{align*}
  \forall S \subseteq U, \quad (\*\lambda * \mu)(S) \propto \*\lambda^{S} \mu(S),
\end{align*}
where $\*\lambda^S = \prod_{i \in S} \lambda_i$.
In particular, if $\*\lambda$ is a constant vector with $\*\lambda_u = \lambda$, we may write $\lambda* \mu$ instead.

The field dynamics $P^{\-{FD}}_\theta$ with parameter $\theta \in (0,1)$ in each round updates a configuration $X \in 2^U$ according to the following rules:
\begin{itemize}
  \item sample $S' \sim \bigotimes_{u \in U} \-{Ber}(\theta)$ and let $S = S' \cup X$;
  \item update $X$ according to distribution $(\theta^{-1} * \mu) \tp{\cdot \mid \+P_{X,U \setminus S}}$,
\end{itemize}
where we recall the event $\+P_{X,Y} = \{T \subseteq U \mid T \cap Y = X \cap Y\}$.

  It has been showed in~\cite{chen2021rapid} that the field dynamics $P^{\-{FD}}_\theta$ has the stationary distribution $\mu$.
  %Moreover, it is proved to be fast mixing when $\*\lambda * \mu$ is spectral independent for all $\*\lambda \in \mathbb{R}^U_{>0}$ (see \cref{lem:field-dynamics-mixing}).
  We note that the only non-trivial step in the field dynamics is to sample from a new distribution $(\theta^{-1} * \mu) \tp{\cdot \mid \+P_{X,U \setminus S}}$.
  %So, assuming the spectral independence requirement on $\mu$, to get a fast sampler, we only need to figure out how to sample from the distribution $(\theta^{-1} * \mu) \tp{\cdot \mid \+P_{X,U \setminus S}}$.
  A key intuition of the field dynamics is that $(\theta^{-1} * \mu) \tp{\cdot \mid \+P_{X,U \setminus S}}$ might be relatively easy to sample from when we choose a good parameter $\theta$.
  Hence, when the field dynamics itself is rapid mixing, it actually reduces the task of sampling from $\mu$ to an easier task of sampling from $(\theta^{-1} * \mu) \tp{\cdot \mid \+P_{X,U \setminus S}}$.
  Our algorithm for the random cluster model is based on this idea where we use a Glauber dynamics to generate approximate samples from $(\theta^{-1} * \mu) \tp{\cdot \mid \+P_{X,U \setminus S}}$ so as to sample from the original distribution (see \cref{alg:field} and \cref{alg:resample} for the details).
  Moreover, We also establish the rapid mixing of field dynamics via the spectral independence.

% \begin{align*}
%   \forall S \subseteq E, \quad \overline{\mu}(S) := \mu(E \setminus S).
% \end{align*}
% Let $\theta \in (0, 1)$, we will use the field dynamics $P^{\-{FD}}_\theta(\overline{\mu})$ for the distribution $\overline{\mu}$.
% Start from $X = \emptyset$, in the $t$-th step, $P^{\-{FD}}_\theta(\overline{\mu})$ updates $X$ as:
% \begin{itemize}
% \item generates $S \sim \bigotimes_{i\in E} \-{Ber}(1 - \theta)$ and let $R = S \cap X$;
% \item updates the state of $S \setminus R$ in $X$ accroding to the distribution of $(\theta * \overline{\mu})(\cdot \mid \*1_R)$.
% \end{itemize}

In recent years, there is a long line of works~\cite{anari2020spectral,chen2020optimal,chen2021rapid,anari2021entropic,blanca2021mixing, anari2021entropicII,chen2022localization,chen2022optimal} establishing the relation between the mixing time of select-update dynamics and the spectral independence.
We first introduce the notion of spectral independence. 
\begin{definition}[influence matrix, \cite{anari2020spectral}]
  Let $\mu$ be a distribution over $2^U$ on ground set $U$ of size $n$. 
  The influence matrix $\Psi^{\mathrm{inf}}_{\mu}$ is an $n$ by $n$ matrix defined as
  \begin{align*}
    \forall i,j \in U ,\quad \Psi^{\mathrm{inf}}_\mu(i,j) = \begin{cases}
    \mu(i \mid j) - \mu(i \mid \overline{j}) & \text{if } i \neq j \text{ and } 0<\Pr[S \sim \mu]{j \in S}<1,\\
    0 & \text{otherwise.}
    \end{cases}
  \end{align*}
\end{definition}

\begin{definition}[spectral independence in infinite norm]\label{def:si-inf}
  Let $\mu$ be a distribution over $2^U$ on ground set $U$. The distribution $\mu$ is $C$-spectrally independent, if 
  \begin{align*}
    \norm{\Psi^{\-{inf}}_{\mu}}_\infty \le C. 
  \end{align*}
  Furthermore, $\mu$ is $C$-spectrally independent under all pinnings, if for any $\tau, \Lambda \subseteq U$ with $\mu(\+P_{\tau,\Lambda})>0$, $\mu(\cdot \mid \+P_{\tau,\Lambda})$ projected on $U \setminus \Lambda$ is $C$-spectrally independent.
\end{definition}
We remark the notion of spectral independence in \Cref{def:si-inf} is stronger than that in~\cite{anari2020spectral}, where the distribution $\mu$ is $C$-spectrally independent, if $\lambda_{\max}(\Psi^{\-{inf}}_{\mu}) \le C$.

In recent progress of high-dimensional expansion and the analysis of Glauber dynamics on anti-ferromagnetic two-state spin systems, the following entropy decay result for field dynamics is established~\cite{anari2021entropicII, chen2021optimalIsing, chen2022localization, chen2022optimal}.
\begin{lemma}[entropy decay of field dynamics]\label{lem:field-dynamics-mixing}
  Let $\mu$ be a distribution over $2^U$ on ground set $U$. If $\*\lambda * \mu$ is $C$-spectrally independent under all pinnings for all $\*\lambda \in \mathbb{R}_{>0}^U$, then for any $\theta \in (0,1)$ and distribution $\nu$ absolutely continuous with respect to $\mu$, let $\kappa = \tp{\theta/\e}^{C + 3}$, it holds that
  \begin{align*}
    D_{\-{KL}}(\nu P^{\-{FD}}_{\theta} \parallel \mu P^{\-{FD}}_\theta) \le (1-\kappa) D_{\-{KL}}(\nu \parallel \mu).
  \end{align*}
\end{lemma}
For completeness, we include the proof of \Cref{lem:field-dynamics-mixing} in \Cref{append:field}.
By~\eqref{eq:mixing-decay}, this entropy decay result implies a tight bound on the mixing time of field dynamics.

%\newpage
\subsection{Field dynamics simulator}\label{sec:simulation}

% By \Cref{lem:field-dynamics-mixing}, the task of proving the rapid mixing of $P^{\-{FD}}_\theta(\overline{\mu})$ is reduced to establish the spectral independence for $\overline{\mu}$.
% By the construction of $\overline{\mu}$, it is equivalent to establish the spectral independence for $\mu$.
% One of our main contribution in this paper is the establishment of the spectral independence under all pinnings and all external fields for random cluster model.
We are now ready to introduce the field dynamics simulator for random cluster model. 
Parameters $\theta$, $T^{\mathrm{FD}}$, and $T^{\mathrm{GD}}$ are to be determined.\\

\begin{algorithm}[H]
  \SetKwInOut{Input}{input}
  \SetKwInOut{Output}{output}

  \Input{Graph $G=(V,E)$, parameters $\*p \in (0,1)^E$, $\*\lambda \in [0,1)^V$ and $\epsilon \in (0,1)$}
  \Output{a random configuration $X \subseteq E$ satisfying $d_{\mathrm{TV}}(X,\mu^{\mathrm{RC}}_{E,\*p,\*\lambda}) < \epsilon$}
  %\eIf{$\abs{V} \le N_0$}{
  %  sample $\mu^{\mathrm{RC}}_{E,\*p,\*\lambda}$ via a brute-force algorithm, i.e. calculate the weight of each configuration and draw a configuration $X$ with probability proportional to its weight.
  %}
  %{
  Initialize $X = E$;\\
  \For{$t \gets 1$ \KwTo $T^{\mathrm{FD}}$}{
    draw $S' \sim \bigotimes_{e \in E} \mathrm{Ber}(\theta)$ and let $S = S' \cup X$;\label{line:draw-S}\\
    update $X \gets \text{Resample}(G'=(V,S),\*p^{\star},\*\lambda,\tp{2T^{\mathrm{FD}}}^{-1} \epsilon)$, where $p^\star_e = \frac{p_e}{\theta(1-p_e)+p_e}, \forall e \in S$.\label{line:update-block} 
  }
  \Return{X}
  %}
  \caption{field dynamics simulator for random cluster model}\label{alg:field}
\end{algorithm}

\begin{algorithm}
  \SetKwInOut{Input}{input}
  \SetKwInOut{Output}{output}

  \Input{Graph $G=(V,E)$, parameters $\*p \in (0,1)^E$, $\*\lambda \in [0,1)^V$, and $\epsilon \in (0,1)$}
  \Output{a random configuration $X \subseteq E$ satisfying $d_{\mathrm{TV}}(X,\mu^{\mathrm{RC}}_{E,\*p,\*\lambda}) < \epsilon$}
    Initialize $X = E$;\\
    \For{$t \gets 1$ \KwTo $T^{\mathrm{GD}}$}{
      draw $e \in E$ uniformly at random;\\
      update $X$ according to $\mu^{\mathrm{RC}}_{E,\*p,\*\lambda}\tp{\cdot \mid \+P_{X,E \setminus \{e\}}}$.\label{line:gd-update}\\ 
    }
    \Return{$X$}
  \caption{Resample$(G$, $\*p$, $\*\lambda$, $\epsilon)$}\label{alg:resample}
\end{algorithm}

\begin{remark}\label{rem:field}
    In \Cref{alg:field}, since $S = S' \cup X$, it holds that $ X \cap (E\setminus S) = \emptyset $.
    This means all the elements in $E\setminus S$ are in the ``out'' state in $X$.
    So, it is straight forward to verify that
  \begin{align*}
    \tp{\theta^{-1} * \mu^{\mathrm{RC}}_{E,\*p,\*\lambda}}\tp{\cdot \mid \+P_{X,E \setminus S}}=\mu^{\mathrm{RC}}_{S,\*p^\star,\*\lambda}(\cdot),
  \end{align*}
  where $p^\star_e = \frac{p_e}{p_e+\theta(1-p_e)}$ for all $e \in S$. 
  Hence, when $\abs{V} > N_0$, \Cref{alg:field} is exactly the field dynamics assuming \Cref{alg:resample} being a perfect sampler.
  Though in our implementation, \Cref{alg:resample} returns approximate samples and causes biases.
  %Therefore, \Cref{alg:field} is indeed a simulation of field dynamics on $\mu^{\mathrm{RC}}_{E,\*p,\*\lambda}$.
\end{remark}

Let $\mu^{\mathrm{RC}}_{E,\*p,\*\lambda}$ be the distribution of the random cluster model specified by graph $G=(V,E)$, parameters $\*p \in (0,1)^E$ and $\*\lambda \in [0,1)^V$. 
Furthermore, let 
\begin{align*}
  p_{\min} = \min_{e \in E} p_e  \text{ and }  \lambda_{\max} = \max_{v \in V} \lambda_{v}.
\end{align*}
We first state the mixing time results for both field dynamics and Glauber dynamics, and then prove \Cref{thm:alg-RC} with these results.

% \begin{lemma} \label{lem:SI-RC}
%   Let $G = (V, E)$ be a graph with $n$ vertices and $m$ edges.
%   Let $\*\lambda \in \mathbb{R}^n, \*p \in \mathbb{R}^m$ be two constant vectors satisfying $0 < \*\lambda, \*p < 1$.
%   Let $\mu$ be the Gibbs distribution of the random cluster model specified by $G$, $\*p$, and $\*\lambda$.
%   Then, for any $\*{\overline{p}} \in \mathbb{R}^m_{> 0}$, $\*{\overline{p}} * \mu$ is $(1 - \lambda_{\max})^{-2}$-spectrally independent under all pinnings.
% \end{lemma}
% \todo{add some remark here}
\begin{lemma}\label{lem:mixing-field}
  The mixing time of the field dynamics initialized from $E$ satisfies
  % Let $(X_t)_{t \in \mathbb{N}}$ be the field dynamics with parameter $\theta \in (0, 1)$, starting from $X_0 = E$ on the distribution $\mu^{\-{RC}}_{E, \*p, \*\lambda}$.
  % Now, suppose 
  \begin{align*}
    \forall \epsilon \in (0,1), \quad T(\epsilon,E) \leq \tp{\frac{\e}{\theta}}^{5(1-\lambda_{\max})^{-2}}\tp{2 \log n + \log \log \frac{2}{p_{\min}} + \log \frac{1}{2\epsilon^2}}.
  \end{align*}
  % then it holds that $\DTV{X_T}{\mu^{\-{RC}}_{E, \*p, \*\lambda}} \leq \epsilon$, where $n$ is the number of vertices.
  
\end{lemma}
This mixing time result is a corollary of \Cref{lem:field-dynamics-mixing} and the spectral independence of $\mu^{\mathrm{RC}}_{E,\*p,\*\lambda}$.
\begin{lemma} \label{lem:SI-RC}
  $\mu^{\mathrm{RC}}_{E,\*p,\*\lambda}$ is $2(1-\lambda_{\max})^{-2}$-spectrally independent under all pinnings.
\end{lemma}
The proof of \Cref{lem:mixing-field} assuming \Cref{lem:SI-RC} is deferred to \Cref{sec:field-proof}.

\begin{remark}
Establishing spectral independence is a challenging task.
A series of works\cite{chen2020rapid,blanca2021mixing, liu2021coupling,chen2021spectral, abdolazimi2021matrix} establish spectral independence via different approaches, including correlation decay, path coupling, stability of polynomial, and the trickle-down phenomenon.
An $O_{\Delta,\*p,\*\eta}(1)$ spectral independence result for the subgraph-world model was established in~\cite{chen2022spectral}.
% {\color{red}
This result has the dependency on $\Delta$ and does not imply a spectral independence bound for the random cluster model.
Our method is quite different from the previous approaches for bounding the spectral independence.
In contrast to previous works, we establish a better spectral independence result for the subgraph-world model that is independent of $\Delta$ by a sophisticated coupling procedure.
This coupling procedure also enables us to lift the spectral independence result from the subgraph-world model to the random cluster model.
As far as we know, \cref{lem:SI-RC} gives the first spectral independence bound for the random cluster model.
The details will be described in \Cref{sec:field}. 
% }

\end{remark}
% The proof of \Cref{lem:SI-RC} is outlined in \Cref{sec:outline-SI-RC}.

% Though we have showed that the field dynamics $P^{\-{FD}}_\theta$ is rapid mixing, there is a non-trivial step in field dynamics, i.e. sampling from the distribution $(\theta * \overline{\mu})(\cdot \mid \*1_R)$, that should be implemented carefully.
% Note that $(\theta * \overline{\mu})(\cdot \mid \*1_R) = ((1/\theta) * \mu)(\cdot \mid \*0_R)$, which is exactly the Gibbs distribution of the random cluster model specified by the graph $(V, E\setminus R)$, $\*p'$, and $\*\lambda$.

% The next tool we use is the Glauber dynamics.
% Suppose $\nu$ is the Gibbs distribution of a random cluster model specified by a graph $(\overline{V}, \overline{E})$, and vectors $\*{\overline{p}}, \*{\overline{\lambda}}$.
% The Glauber dynamics $P^{\-{GD}}(\nu)$ for $\nu$ start with $X = \overline{E}$, then in each step, it updates $X$ as:
% \begin{itemize}
% \item pick an edge $e \in \overline{E}$ uniformly at random;
% \item add $e$ to $X$ with probability $\frac{\nu(X \cup \{e\})}{\nu(X \cup \{e\}) + \nu(X \setminus \{e\})}$, otherwise remove $e$ from $X$.
% \end{itemize}

% Our next contribution shows that when $\theta$ is sufficiently small, which ensures $p'_e \geq 1 - K/\log n$ for some constant $K$, then the Glauber dynamics on $((1/\theta) * \mu)(\cdot \mid \*0_R)$ is rapid mixing.

\begin{lemma} \label{lem:GD}
  % Let $G = (V, E)$ be a graph with $n$ vertices and $m$ edges.
  % Let $\*\lambda \in \mathbb{R}^n, \*p \in \mathbb{R}^m$ be two constant vectors satisfying $0 < \*\lambda < 1$ and $1 - K/\log n \leq \*p < 1$.
  % Let $\mu$ be the Gibbs distribution of the random cluster model specified by $G$, $\*p$, and $\*\lambda$.
  % Let $T \in \mathbb{N}$ and $X_T$ be the state of $P^{\-{GD}}(\mu)$ after $T$ steps.
  % Let $\epsilon \in (0, 1)$ be a real number.
  % Let $(X_t)_{t \in \mathbb{N}}$ be the Glauber dynamics starting from $X_0 = E$ on distribution $\mu^{\mathrm{RC}}_{E,\*p,\*\lambda}$.
  If it holds that
  \begin{align*}
    \tp{1-p_{\min}}\log n \leq \min\left\{10^{-7}, \frac{1-\lambda_{\max}}{27} \right\},
  \end{align*}
  then the mixing time of the Glauber dynamics satisfies
  \begin{align*}
    \forall \epsilon \in (0,1), \quad T_{\-{mix}}(\epsilon) \le 25m \log m \log \frac{1}{\epsilon} + 1.
  \end{align*}
  % then it holds that $\DTV{X_T}{\mu} \leq \epsilon$, where $T = \lceil 2m(\log m + \log(2/\epsilon)) \rceil$.
\end{lemma}
The proof of \Cref{lem:GD} involves a coupling with stationary argument, which will be presented in \Cref{sec:good-regime}.
We are now ready to prove \Cref{thm:alg-RC}.

\begin{proof}[Proof of \Cref{thm:alg-RC}]
  We choose parameters in \Cref{alg:field} and \Cref{alg:resample} as follows:
  
  \begin{center}
    \begin{tabular}[h]{c|c}
      \hline
      parameter & value \\
      \hline
      $\theta$ & $\frac{\min \left\{10^{-7}, \frac{1-\lambda_{\max}}{27}\right\} p_{\min}}{\log n}$ \\
      $T^{\-{FD}}$ & $\left\lceil \tp{\frac{\e}{\theta}}^{5(1-\lambda_{\max})^{-2}}\tp{2 \log n + \log \log \frac{2}{p_{\min}} + \log \frac{2}{\epsilon^2}} \right\rceil$ \\
      % $T^{\-{GD}}$ & $\left\lceil 2m(\log m + \log(2T^{\mathrm{FD}}/\epsilon)) \right\rceil$ \\
      $T^{\-{GD}}$ & $\left\lceil 25m \log m \log \frac{2T^{\mathrm{FD}}}{\epsilon} \right\rceil + 1$ \\
      %$N_0$ & $10$ \\
      \hline
    \end{tabular}
  \end{center}
 % \begin{align*}
 %   \theta = \frac{K p_{\min}}{\log n}, T^{\mathrm{FD}} = \left\lceil \tp{\frac{\e}{\theta}}^{5(1-\lambda_{\max})^{-2}}\tp{2 \log n + \log \log \frac{1}{\delta} + \log \frac{2}{\epsilon^2}} \right\rceil,\\
 %   N_0 = \max \left\{\exp\tp{12(1-\lambda_{\max})^{-2}},\frac{3}{\delta},\log \tp{\frac{2}{\delta^2}}\right\} \text{ and } T^{\mathrm{GD}} = \lceil 2m(\log m + \log(2T^{\mathrm{FD}}/\epsilon)) \rceil,
 % \end{align*}
 % \todo{format issue}
  % where we use $K = 10^{-14} \epsilon^2 \exp\tp{-28\;(1-\lambda_{\max})^{-2}}$.

  % When $\abs{V} \leq N_0$, it could be easily checked that our algorithm returns perfect samples according to the distribution $\mu^{\-{RC}}_{E, \*p, \*\lambda}$ while the overall running time is bounded by $2^{O(N_0^2)} = \-{poly}(1/\epsilon)$.
  
  % In the rest part of the proof, we assume that $\abs{V} > N_0$.
  To prove \Cref{thm:alg-RC}, it suffices to show that:
  \begin{enumerate}
  %\item\label{req:item-1}
    %For each $S \subseteq E$, the sample $X$ returned from Resample$((V,S), \*p^\star, \*\lambda, (2T^{\-{FD}})^{-1}\epsilon)$ in \Cref{line:update-block} of \Cref{alg:field} satisfies $d_{\mathrm{TV}}(X,\mu^{\mathrm{RC}}_{S,\*p^\star,\*\lambda}) < \frac{\epsilon}{2T^{\mathrm{FD}}}$.
    \item\label{req:item-1} The sample $X$ returned from \Cref{alg:field} satisfies $d_{\mathrm{TV}}(X,\mu^{\mathrm{RC}}_{E,\*p,\*\lambda}) < \epsilon$.
    \item\label{req:item-2} The overall running time can be achieved in $m \tp{\epsilon^{-1} \cdot \delta_p^{-1} \cdot \log n}^{O\tp{\delta_\lambda^{-2}}}$.
  \end{enumerate}
  %Requirement~\eqref{req:item-2} directly follows from \Cref{lem:field-dynamics-mixing}, \Cref{rem:field}. 
  %Hence, it only remains to verify requirement~\eqref{req:item-1} and requirement~\eqref{req:item-3}.

  First, we prove requirement~\eqref{req:item-1}.
  For convenience, let $\widehat{P^{\-{FD}}_{\theta}}$ denote the transition matrix of the simulation algorithm (\Cref{alg:field}) and $\mu = \mu^{\-{RC}}_{E, \*p, \*\lambda}$.
  % We run $\widehat{P^{\-{FD}}_\theta}$ for $T^{\-{FD}}$ iterations and get $\emptyset = X_0, X_1, X_2, \cdots, X_{T^{\-{FD}}}$ and suppose the field dynamics $P^{\-{FD}}_\theta$ also runs for $T^{\-{FD}}$ steps and gets $\emptyset = Y_0, Y_1, Y_2, \cdots, Y_{T^{\-{FD}}}$.
    %  In the $i$-th step, $X_i, Y_i$ are sampled from the optimal coupling from $\widehat{P^{\-{FD}}_{\theta}}(X_{i-1}, \cdot)$ and $P^{\-{FD}}_\theta(Y_{i-1}, \cdot)$.
   Let $(X_t)_{0 \le t \le T^{\-{FD}}}$ (respectively, $(Y_t)_{0 \le t \le T^{\-{FD}}}$) be the Markov chain starting from configuration $X_0 = E$ (respectively, $Y_0=E$) with transition matrix $\widehat{P^{\-{FD}}_{\theta}}$ (respectively, $P^{\-{FD}}_\theta$).
   It holds that
   \[
     \DTV{X_{T^{\-{FD}}}}{\mu}
     \leq \DTV{X_{T^{\-{FD}}}}{Y_{T^{\-{FD}}}} + \DTV{Y_{T^{\-{FD}}}}{\mu}
     \leq \Pr{X_{T^{\-{FD}}} \neq Y_{T^{\-{FD}}}} + \DTV{Y_{T^{\-{FD}}}}{\mu}.
   \]
   Note that $X_{T^{\-{FD}}} \neq Y_{T^{\-{FD}}}$ is equivalent to that there exists $1 \leq i \leq T^{\-{FD}}$ such that $X_i \neq Y_i$ but $X_j = Y_j$ for all $j < i$.
  %  By union bound, we have
  Hence, 
  \begin{align*}
     \Pr{X_{T^{\-{FD}}} \neq Y_{T^{\-{FD}}}}
     &= \sum_{i=1}^{T^{\-{FD}}} \Pr{X_i \neq Y_i \text{ and } \forall j < i, X_j = Y_j} 
     \leq \sum_{i=1}^{T^{\-{FD}}} \Pr{X_i \neq Y_i \mid X_{i-1} = Y_{i-1}}.
     %\leq \sum_{i=1}^{T^{\-{FD}}} \Pr{X_i \neq Y_i \mid \forall j < i, X_j = Y_j},
     %(\text{Markov chain}) \quad
     %&= \sum_{i=1}^{T^{\-{FD}}} \Pr{X_i \neq Y_i \mid X_{i-1} = Y_{i-1}},
   \end{align*}
    For any $1 \le i \le T^{\-{FD}}$, consider the following coupling of $X_i$ and $Y_i$ condition on $X_{i-1} = Y_{i-1}$:
    \begin{enumerate}
      \item Generate set $S' \sim \bigotimes_{u \in U} \mathrm{Ber}(\theta)$ and $S = S'\cup X_{i-1}$, i.e., the first step of the field dynamics;
      \item Generate $X_i$ and $Y_i$ according to the optimal coupling of distribution $\mu^{\-{RC}}_{S,\*p^\star,\*\lambda}$ and $\widetilde{\mu^{\-{RC}}_{S,\*p^\star,\*\lambda}}$, where $\widetilde{\mu^{\-{RC}}_{S,\*p^\star,\*\lambda}}$ is the distribution generated by Resample$((V, S), \*p^\star, \*\lambda, (2T^{\-{FD}})^{-1}\epsilon)$.
    \end{enumerate}
  %  A set $S \subseteq E$ is generated first, then $X_i = A(S)$ is generated by Resample$((V, S), \*p^\star, \*\lambda, (2T^{\-{FD}})^{-1}\epsilon)$ and $Y_i$ is generated according to the distribution $\mu^{\-{RC}}_{S, \*p^\star, \*\lambda}$.
  %  Consider the one-step optimal c It holds that
  Hence, we have
  \begin{align*}
     \Pr{X_i \neq Y_i \mid X_{i-1} = Y_{i-1}} &\leq \max_{S\subseteq E} \DTV{\widetilde{\mu^{\-{RC}}_{S,\*p^\star,\*\lambda}}}{\mu^{\-{RC}}_{S, \*p^\star, \*\lambda}} =: \epsilon'.
   \end{align*}
   Note that $(1-p^\star_{\min}) \log n \le \frac{\theta \log n}{p_{\min}} \le \min \left\{10^{-7}, \frac{1-\lambda_{\max}}{27}\right\}$. By \Cref{lem:GD} and our choice of $T^{\mathrm{GD}}$, $\epsilon' \leq \frac{\epsilon}{2T^{\-{FD}}}$. Therefore, by \Cref{lem:mixing-field} and our choice of $T^{\-{FD}}$, it holds that 
   \begin{align*}
   \DTV{X_{T^{\-{FD}}}}{\mu} \leq T^{\-{FD}}\epsilon' + \epsilon/2 \leq \epsilon.
   \end{align*}
   This proves the requirement~\eqref{req:item-1}.

  Lastly, we verify requirement~\eqref{req:item-2}. 
  In each iteration of the field dynamics in \Cref{alg:field}, we need $O(m)$ time to generate the set $S$.
  This consumes $mT^{\-{FD}}$ time.
  Besides, the algorithm needs $T^{\mathrm{FD}}$ iterations of $T^{\mathrm{GD}}$-steps Glauber dynamics starting from configuration $X=S$ on distribution $\mu^{\mathrm{RC}}_{S,\*p^\star,\*\lambda}$ for some $S \subseteq E$.
  We claim that each step of the Glauber dynamics could be implemented in $\-{polylog}(n)$ time, then the total running time is
  \begin{align} \label{eq:time-raw}
    mT^{\-{FD}} + T^{\-{FD}} \cdot T^{\-{GD}} \cdot \-{polylog(n)}.
  \end{align}
  By our choice of $T^{\-{FD}}$ and $T^{\-{GD}}$, it holds that
  \begin{align*}
    T^{\-{FD}} = \tp{(1-\lambda_{\max})^{-1} \cdot p_{\min}^{-1} \cdot \log n}^{O((1-\lambda_{\max})^{-2})} \text{ and } T^{\-{GD}} \leq 25 m \log m \log \frac{2T^{\-{FD}}}{\epsilon} + 2.
  \end{align*}
  Together with $\lambda_{\max} \le 1-\delta_{\lambda}$ and $p_{\min} \ge \delta_{\beta}$, the total running time~\eqref{eq:time-raw} could be bounded by
  \begin{align*}
    m \tp{\delta_\lambda^{-1} \cdot \delta_p^{-1} \cdot \log n}^{O\tp{\delta_\lambda^{-2}}}.
  \end{align*}

  We only left to show that each step of the Glauber dynamics could be implemented in $\-{polylog}(n)$ time.
  Suppose the current configuration is $X$,
  the Glauber dynamics will first use $O(\log n)$ time to draw a random edge $e = (u, v) \in E$.
  Let $C_u = C_u(X)$ and $C_v = C_v(X)$ be the connected components in graph $(V,X \setminus \{e\})$, containing $u$ and $v$, respectively.
  Then the probability $p_{X,e}$ that $X$ will be updated by $X \cup \{e\}$ is
  \begin{align}\label{eq:rc-transition}
    p_{X,e} = 
    \begin{cases}
      p & \text{if } C_u = C_v, \\
      \frac{1+\*\lambda^{C_u \cup C_v}}{1+\*\lambda^{C_u \cup C_v} + (1-p^\star_e)\tp{\*\lambda^{C_u} + \*\lambda^{C_v}}} & \text{otherwise},
    \end{cases}
  \end{align}
  where $\*\lambda^R = \prod_{i \in R} \lambda_i$ for $R \subseteq V$. 
  In order to calculate $p_{X, e}$ fast, we need a data structure that supports the following operations:
  %Hence, it remains to design an algorithm supporting the following updates and queries in $O(\mathrm{polylog}(n))$ time.
  \begin{itemize}
    \item update $X\gets X \cup \{e\}$ for any given $e \in E$;
    \item update $X\gets X \setminus \{e\}$ for any given $e\in E$;
    \item query if $C_u(X) = C_v(X)$ for any given $u,v \in V$;
    \item query $\*\lambda^{C_u(X)}$ for any given $u \in V$.
  \end{itemize}
  These updates and queries can all be handled in $O(\log^2 n)$ amortized time by the data structure in \cite[Section 3]{wulff2013faster}. This concludes the proof of requirement~\eqref{req:item-2} and \Cref{thm:alg-RC}.
\end{proof}

% \subsection{Spectral independence for random cluster model} \label{sec:outline-SI-RC}

% \subsection{Rapid mixing of Glauber dynamics in good regime} \label{sec:outline-GD}

%\todo{Proof outline can be put behind prelim}

\section{Spectral independence of random cluster model}\label{sec:field}
In this section, we are going to prove \Cref{lem:SI-RC}. We prove \Cref{lem:SI-RC} via the following lemma. 
\begin{lemma} \label{lem:inf-RC}
  Let $\mu$ be the distribution of the random cluster model specified by graph $G = (V, E)$, parameters $\*p \in [0, 1]^E$ and $\*\lambda \in [0, 1)^V$.
  Then, $\mu$ is $2(1 - \lambda_{\max})^{-2}$-spectrally independent.
\end{lemma}
\begin{proof}[Proof of \Cref{lem:SI-RC}]
  For any $\tau, \Lambda \subseteq E$, define $\tilde{\*p} \in [0,1]^E$ by
  \begin{align*}
    \forall e \in E, \quad \tilde{p}_e =
    \begin{cases}
      0 & \text{if }e \in \Lambda \setminus \tau,\\
      1 & \text{if }e \in \Lambda \cap \tau,\\
      p_e & \text{if }e \in E \setminus \Lambda.
    \end{cases}
  \end{align*}  
  Note that $\mu^{\mathrm{RC}}_{E,\tilde{\*p},\*\lambda}$ is exactly $\mu(\cdot \mid \+P_{\tau,\Lambda})$. This concludes the proof of \Cref{lem:SI-RC}.  
\end{proof}

In order to prove \Cref{lem:inf-RC}, we introduce a simple coupling criteria for spectral independence. 
\begin{definition}[coupling independence]
  A distribution $\mu$ over $2^{E}$ on ground set $E$ is $C$-\ci, if for all $i \in E$, there exists a coupling $(X,Y)$ of distribution $\mu(\cdot \mid i)$ and $\mu(\cdot \mid \overline{i})$, that
  \begin{align*}
    \E{\abs{X \oplus Y}} \le C.
  \end{align*}
  Furthermore, a distribution is $C$-\ci{} under all pinnings, if for any $\Lambda,\tau \subseteq U$ with $\mu(\+P_{\tau,\Lambda}) > 0$, $\mu(\cdot \mid \+P_{\tau,\Lambda})$ projected on $U \setminus \Lambda$ is $C$-\ci.
\end{definition}

\begin{proposition}\label{ci-implies-si}
  If a distribution $\mu$ over $2^E$ is $C$-\ci{}, then $\mu$ is $C$-spectrally independent.
\end{proposition}
\begin{proof}
  Fix $i \in E$. Let $(X,Y)$ be a coupling of $\mu(\cdot \mid i)$ and $\mu(\cdot \mid \overline{i})$ such that $\E{\abs{X \oplus Y}} \le C$, then
\begin{align*}
  \sum_{j \in E \setminus \{i\} } \abs{\mu(j \mid i) - \mu(j \mid \overline{i})}  
  \le \sum_{j \in E} \E{\*1[X_j \neq Y_j]}
  = \E{\abs{X \oplus Y}} \le C,
\end{align*}
where the first inequality holds by standard coupling lemma.
Therefore,
\begin{align*}
  \norm{\Psi^{\mathrm{inf}}_{\mu}}_{\infty} &= \max_{i \in E} \sum_{j \in E \setminus \{i\} } \abs{\mu(j \mid i) - \mu(j \mid \overline{i})}  \le C. \qedhere
\end{align*}
\end{proof}

Now, to prove \Cref{lem:inf-RC}, we first claim the coupling independence for the distribution of subgraph-world model.
\begin{lemma}\label{lm:sg-coupling}
  Let $\nu$ be the distribution of subgraph-world model specified by graph $G = (V, E)$, and vectors $\*p \in [0, \frac{1}{2}]^{E}$, $\*\eta \in (0, 1]^V$.
  It holds that $\nu$ is $\frac{1}{2\eta_{\min}^2}$-coupling independent.
\end{lemma}

Then, we show that once we have a coupling of distributions of generalized subgraph-world model, we could ``lift'' it to the random cluster model.

\begin{lemma}\label{lm:lifting-no-pin}
  Let $G=(V,E)$ be an undirected graph, $\*p \in [0,1]^E$ and $\*\lambda \in [0,1)^V$ be parameters.
  Let $\mu$ be the distribution of a random cluster model specified by graph $G$, parameters $\*p$ and $\*\lambda$.
  Let $\nu$ be the distribution of a subgraph-world model specified by graph $G$, parameters $\frac{\*p}{2} = \tp{\frac{p_e}{2}}_{e \in E}$ and $\*\eta = \tp{\frac{1-\lambda_v}{1+\lambda_v}}_{v \in V}$.
  
  If $\nu$ is $C$-\ci{}, then $\mu$ is also $C$-\ci{}.
\end{lemma}

\Cref{lem:inf-RC} is proved by combining \Cref{ci-implies-si}, \Cref{lm:sg-coupling}, and \Cref{lm:lifting-no-pin}.

The proof of \Cref{lm:sg-coupling} and \Cref{lm:lifting-no-pin} are given in \Cref{sec:sg-coupling} and \Cref{sec:lift} respectively.

\subsection{Coupling independence of generalized subgraph-world model} \label{sec:sg-coupling}
In this section, we prove \Cref{lm:sg-coupling}.
For convenience, for $\sigma \in \mathbb{R}^V$, we use $\as{\sigma}{u}{c}$ to denote a vector $\sigma$ with $\sigma_u$ being changed to value $c$. 
Meanwhile, for $\sigma, \tau \in \{0, 1\}^V$, we use $\sigma \oplus \tau$ to denote the bitwise exclusive or of $\sigma$ and $\tau$.
% We could represent $\nu$ in the form of generalized subgraph-world model as $\mu^{\-{GSW}}_{E, \*p, \*\eta, \*1}$.
We now prove a generalized version of \Cref{lm:sg-coupling}.
%We state it as follow.
\begin{lemma} \label{lem:cp-sg}
  Let $G=(V,E)$ be an undirected graph, $\sigma \in \{0,1\}^V$ be parity constraints on vertices, $\*p \in \left[0,\frac{1}{2}\right]^E$ and $\*\eta \in (0,1]^V$ be parameters.
  For any $u \in V$, there is a coupling $(X, Y)$ between $\mu_{E, \*p, \*\eta, \sigma}^{\-{GSW}}$ and $\mu_{E, \*p, \*\eta, \sigma \oplus \*1_u}^{\-{GSW}}$ such that $\E{\abs{X \oplus Y}} \leq \frac{1}{4\eta_{\min}^2}$.
\end{lemma}
We now prove \Cref{lm:sg-coupling}. The proof follows from a standard coupling argument.
\begin{proof}[Proof of \Cref{lm:sg-coupling}]
  Fix $e \in (u, v) \in E$ and let $\nu = \mu^{\mathrm{GSW}}_{E,\*p,\*\eta,\sigma}$. By definition,
  \begin{align*}
    \nu(\cdot \mid \overline{e}) = \mu^{\-{GSW}}_{E\setminus\{e\}, \*p, \eta, \*1} \quad \text{and} \quad \nu(\cdot \mid e) = \mu^{\-{GSW}}_{E\setminus \{e\}, \*p, \eta, \*1 \oplus \*1_u \oplus \*1_v},
  \end{align*}
  %where we use $\*p_{-e}$ to denote $\*p$ with $e$-th element being removed.
  Consider an intermediate distribution
  % \begin{align*}
    $\widetilde{\nu} := \mu^{\-{GSW}}_{E\setminus \{e\}, \*p, \eta, \*1 \oplus \*1_u}$.
  % \end{align*}
  By \Cref{lem:cp-sg}, there are couplings $\+C_1$ of $\nu(\cdot \mid \overline{e})$ and $\widetilde{\nu}$ as well as $\+C_2$ of $\widetilde{\nu}$ and $\nu(\cdot \mid e)$ satisfying
  \begin{align*}
    \forall i \in \{1,2\}, \quad \E[(X,Y) \sim C_i]{\abs{X \oplus Y}} \le \frac{1}{4\eta_{\min}^2}.
  \end{align*}
  Using $\+C_1$ and $\+C_2$, we could construct a coupling $(X, Y)$ of $\nu(\cdot \mid \overline{e})$ and $\nu(\cdot \mid e)$ by: (1) sampling $X \sim \nu(\cdot \mid \overline{e})$; (2) sampling $Z$ proportional to $\+C_1(X, \cdot)$; (3) sampling $Y$ proportional to $\+C_2(Z, \cdot)$.
  It could be verified from the definition of $\+C_1$ and $\+C_2$ that $X, Y$ have correct marginals.
  Again, by \Cref{lem:cp-sg}, it holds that
  \begin{align*}
    \E{\abs{X \oplus Y}}
    &\leq \E{\abs{X\oplus Z} + \abs{Z \oplus Y}} \leq \E{\abs{X\oplus Z}} + \E{\abs{Z \oplus Y}} \leq \frac{1}{2\eta_{\min}^2},
  \end{align*}
  where the last inequality follows from the fact that $(X, Z) \sim \+C_1$ and $(Z, Y) \sim \+C_2$.
\end{proof}

The rest part of this section is dedicated to the proof of \Cref{lem:cp-sg}.
We construct the coupling $(X, Y)$ using the procedure Couple$(G,\*p,\*\eta,\sigma,u,U)$ in \Cref{alg:coupling}, where $U$ denotes the set of visited vertices, and is initialized to $\emptyset$. \Cref{fig:coupling} is an illustration of \Cref{alg:coupling}.

\begin{figure}
  \begin{tikzpicture}[gnode/.style={circle, draw=green!60, fill=green!30, very thick, minimum size=5mm},
    rnode/.style={circle, draw=red!60, fill=red!30, very thick, minimum size=5mm},
    bnode/.style={circle, draw=blue!60, fill=blue!30, very thick, minimum size=5mm},
    rect/.style={rectangle, draw=cyan!40, fill=cyan!5, ultra thick, rounded corners, minimum width=40mm, minimum height=40mm}
    ] 
    \node (rect1) at (2,2.25) [draw=cyan!40, fill=cyan!5, ultra thick, minimum width=5cm,minimum height=2.5cm, rounded corners] {};
    \node[anchor=south] at (rect1.south) {$U=X=Y=\emptyset$};
    \node[rnode] (1) at (0,2){$v_1$}; 
    \node[gnode] (2) at (2,2){$v_2$}; 
    \node[bnode] (3) at (4,2){$v_3$}; 
    \begin{scope}[>={Stealth[black]},
      every node/.style={fill=cyan!5,circle},
      every edge/.style={draw=black,very thick}]
    \path[dotted, very thick] (1) edge node {$e_1$} (2);
    \path[dotted, very thick] (2) edge node {$e_2$} (3);
    \path[dotted, very thick] (3) edge[bend right=40] node {$e_3$} (1);
    \end{scope}
  
    \node (rect2) at (-1.5,-0.75) [draw=green!40, fill=green!5, ultra thick, minimum width=5cm,minimum height=2.5cm, rounded corners] {};
    \node[anchor=south] at (rect2.south) {$U=\{v_1\}$, $X=Y=\{e_2\}$};
    \node[bnode] (1) at (-3.5,-1){$v_1$}; 
    \node[gnode] (2) at (-1.5,-1){$v_2$}; 
    \node[bnode] (3) at (0.5,-1){$v_3$}; 
    \begin{scope}[>={Stealth[black]},
      every node/.style={fill=green!5,circle},
      every edge/.style={draw=black,very thick}]
    \path[very thick] (2) edge node {$e_2$} (3);
    \end{scope}
  
    \node (rect3) at (5.5,-0.75) [draw=red!40, fill=red!5, ultra thick, minimum width=5cm,minimum height=2.5cm, rounded corners] {};
    \node[anchor=south] at (rect3.south) {$U=\{v_1\}$, $X=Y=\emptyset$};
    \node[rnode] (1) at (3.5,-1){$v_1$}; 
    \node[gnode] (2) at (5.5,-1){$v_2$}; 
    \node[bnode] (3) at (7.5,-1){$v_3$}; 
    \begin{scope}[>={Stealth[black]},
      every node/.style={fill=red!5,circle},
      every edge/.style={draw=black,very thick}]
      \path[dotted, very thick] (1) edge node {$e_1$} (2);
      \path[dotted, very thick] (2) edge node {$e_2$} (3);
      \path[dotted, very thick] (3) edge[bend right=40] node {$e_3$} (1);
    \end{scope}
  
    \node (rect4) at (2,-3.75) [draw=cyan!40, fill=cyan!5, ultra thick, minimum width=5cm,minimum height=2.5cm, rounded corners] {};
    \node[anchor=south] at (rect4.south) {$U=\{v_1\}$, $X=Y=\{e_1\}$};
    \node[rnode] (1) at (0,-4){$v_1$}; 
    \node[bnode] (2) at (2,-4){$v_2$}; 
    \node[bnode] (3) at (4,-4){$v_3$}; 
    \begin{scope}[>={Stealth[black]},
      every node/.style={fill=cyan!5,circle},
      every edge/.style={draw=black,very thick}]
      \path[very thick] (1) edge node {$e_1$} (2);
      \path[dotted, very thick] (2) edge node {$e_2$} (3);
      \path[dotted, very thick] (3) edge[bend right=40] node {$e_3$} (1);
    \end{scope}
    \node (rect5) at (9,-3.75) [draw=red!40, fill=red!5, ultra thick, minimum width=5cm,minimum height=2.5cm, rounded corners] {};
    \node[anchor=south] at (rect5.south) {$U=\{v_1\}$, $X=\emptyset, Y=\{e_1\}$};
    \node[gnode] (1) at (7,-4){$v_1$}; 
    \node[rnode] (2) at (9,-4){$v_2$}; 
    \node[bnode] (3) at (11,-4){$v_3$}; 
    \begin{scope}[>={Stealth[black]},
      every node/.style={fill=red!5,circle},
      every edge/.style={draw=red,very thick}]
      \path[very thick] (1) edge node {$e_1$} (2);
    \end{scope}
    \begin{scope}[>={Stealth[black]},
      every node/.style={fill=red!5,circle},
      every edge/.style={draw=black,very thick}]
      \path[dotted, very thick] (2) edge node {$e_2$} (3);
      \path[dotted, very thick] (3) edge[bend right=40] node {$e_3$} (1);
    \end{scope}
    \path [->] (rect1) edge node[left,xshift=0pt,yshift=1pt]{\tiny $r \le q_0$ or $r \ge q_1$} (rect2);
    \path [->] (rect1) edge node[right,xshift=0pt,yshift=1pt]{\tiny $q_0< r <q_1$} (rect3);
    \path [->] (rect3) edge node[left,xshift=0pt,yshift=1pt]{\tiny $X_1 = Y_1$} (rect4);
    \path [->] (rect3) edge node[right,xshift=0pt,yshift=1pt]{\tiny $X_1 \neq Y_1$} (rect5);
    \end{tikzpicture} 
    \caption{
      This is an illustration of \Cref{alg:coupling}. 
      Here, vertices $v_i$ is colored red if $u = v_i$, blue if $\sigma_{v_i} = 1$ and green if $\sigma_{v_i} = 0$.
      Moreover, edge that has not been revealed is represented as dotted line, edge in both $X$ and $Y$ is colored with black, edge in exactly one of $X$ and $Y$ is colored with red, and edge that has been revealed but not in either $X$ or $Y$ is removed.
      }\label{fig:coupling}
  \end{figure}
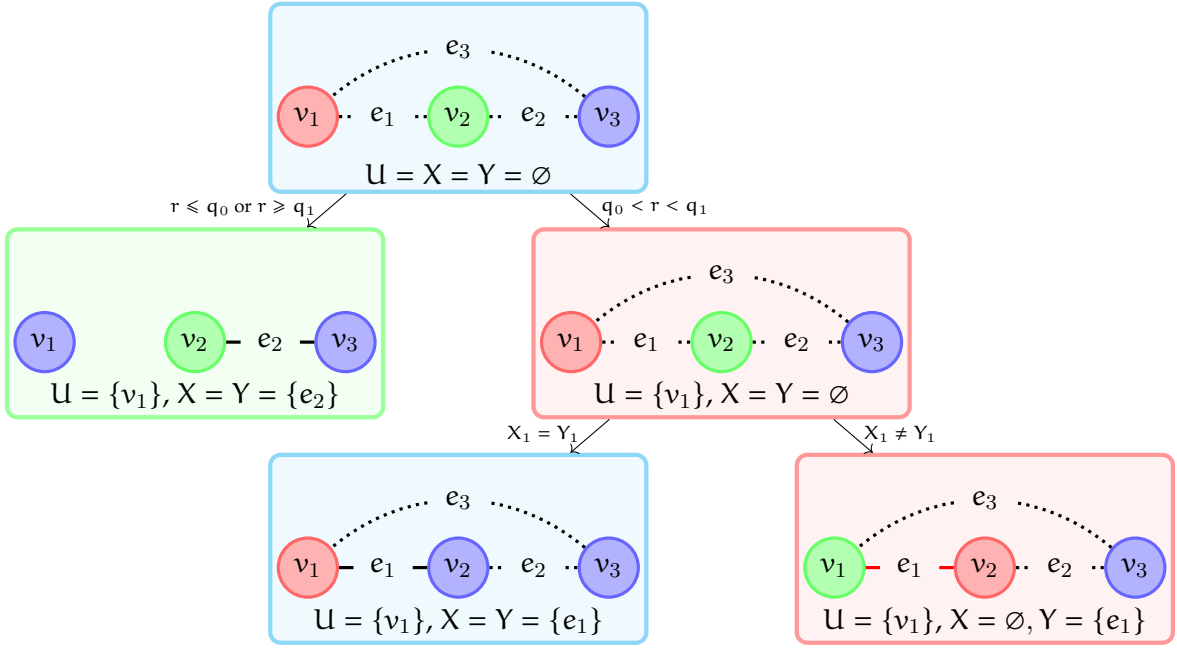
In order to prove \Cref{lem:cp-sg}, it suffices to prove the following properties.
\begin{enumerate}
\item\label{req:termination} 
  Each recursive call in \Cref{alg:coupling} is valid.
\item\label{req:dist-req} $\E{\abs{X \oplus Y}} \le \frac{1}{4\eta_{\min}^2}$;
\item\label{req:coupling-req} $(X,Y) = \text{Couple}(u,E,\*\eta,\sigma,\emptyset)$ is indeed a coupling of $\mu^{\mathrm{GSW}}_{E,\*p,\*\eta,\sigma}$ and $\mu^{\mathrm{GSW}}_{E,\*p, \*\eta, \sigma \oplus \*1_u}$, that is 
  \begin{align*}
    X \sim \mu^{\mathrm{GSW}}_{E,\*p,\*\eta,\sigma} \text{ and } Y \sim \mu^{\mathrm{GSW}}_{E,\*p, \*\eta, \sigma \oplus \*1_u}.
  \end{align*}
\end{enumerate}

\begin{algorithm}
  \caption{Couple$(G,\*p,\*\eta,\*\sigma,u,U)$}\label{alg:coupling}
  \SetKwInOut{Input}{input}
  \SetKwInOut{Output}{output}

  \label{alg:input}\Input{graph $G=(V,E)$, parameters $\*p \in [0,1]^E, \*\eta \in [0,1]^V, \sigma \in \{0,1\}^V$, vertex $u \in V$ and set of visited vertices $U$ satisfying\\
  \begin{enumerate}
    \item $\eta_u = 0$ if and only if $u \in U$;
    \item $Z^{\mathrm{GSW}}_{E,\*p,\*\eta,\sigma},Z^{\mathrm{GSW}}_{E,\*p,\*\eta,\sigma \oplus \*1_u}>0$.
  \end{enumerate}}
  \Output{a pair of random configuration $(X,Y) \in 2^E \times 2^E$.}

  \If{$u \not\in U$}{\label{line:outer-if}
    update $U \leftarrow U \cup \{u\}$;\\
    let $\+A=\{S \subseteq E \mid \abs{S \cap E_u} \equiv 0 \mod 2\}$; \\
    let $R = \tp{\sum_{S \in \+A} \mu^{\mathrm{GSW}}_{E,\*p,\as{\*\eta}{u}{1},\sigma}(S)} \tp{\sum_{S \not\in \+A} \mu^{\mathrm{GSW}}_{E,\*p,\as{\*\eta}{u}{1},\sigma}(S) }^{-1}$, 
    $q_0 = \frac{\eta_u R}{\eta_u R+1}$ and $q_1=\frac{R}{R+\eta_u}$;\\ \label{line:def}
    draw $r \sim \mathrm{Uniform}(0,1)$;\\ \label{line:sample-r}
    update $\eta_u \leftarrow 0$;\\
    \If{$r \ge q_1$}{
      sample $C \sim \mu^{\mathrm{GSW}}_{E,\*p,\*\eta,\as{\sigma}{u}{0}}$;\\
      \Return{$(X,Y) = (C,C)$}
    }
    \If{$r \le q_0$}{
      sample $C \sim \mu^{\mathrm{GSW}}_{E,\*p,\*\eta,\as{\sigma}{u}{1}}$;\\
      \Return{$(X,Y) = (C,C)$}\label{line:after-return}
    }
  }
  pick an arbitrary $e=(u,v) \in E_u$; \\
  let $\nu$, $\pi$ be the distributions of $\mu^{\mathrm{GSW}}_{E,\*p,\*\eta,\sigma}$ and $\mu^{\mathrm{GSW}}_{E,\*p,\*\eta,\sigma \oplus \*1_u}$ projected on $e$ respectively;\\
  sample $(X_1,Y_1)$ from an optimal coupling of $\nu$ and $\pi$;\\
  \If{$X_1 = \{e\}$}{
    update $\sigma \leftarrow \sigma \oplus \*1_u \oplus \*1_v$;\\
  }
  \eIf{$X_1 = Y_1$}{
    $(X_2,Y_2) \leftarrow \text{Couple}((V, E \setminus \{e\}), \*p, \*\eta, \sigma, u, U)$;\\
    \Return{$(X,Y)= (X_1\cup X_2, Y_1 \cup Y_2)$}
  }
  {
    $(X_2,Y_2)\leftarrow \text{Couple}((V, E \setminus \{e\}), \*p, \*\eta, \sigma, v, U)$;\\
    \Return{$(X,Y)= (X_1 \cup X_2, Y_1 \cup Y_2)$}
  }
\end{algorithm}

The first property can be verified easily. We now prove Property (\ref{req:dist-req}) with following observations.
\begin{proposition}\label{prop:geometric}
  Let $U$ be the set of visited vertices upon termination. For any $k \ge 1$,
  \begin{align*}
    \Pr{\abs{U} \ge k} \le \tp{\frac{1-\eta_{\min}}{1+\eta_{\min}}}^{k-1}.
  \end{align*}
\end{proposition}
\begin{proof}
  Note that $\abs{U} \ge k$ implies that first $k-1$ random numbers $r_1,r_2,\ldots,r_{k-1}$ drawn in \Cref{line:sample-r} all lie in their corresponding segments. 
  Therefore,
  \begin{align*}
    \Pr{\abs{U} \ge k} &\le \tp{\max_{R>0, u\in V} \left\{ \frac{R}{R+\eta_u} - \frac{\eta_u R}{\eta_u R+1} \right\} }^{k-1} \le \tp{\frac{1-\eta_{\min}}{1+\eta_{\min}}}^{k-1}. \qedhere
  \end{align*}
\end{proof}

\begin{proposition}\label{prop:edge-bound}
  Let $U$ be the set of visited vertices upon termination and $(X,Y)$ be the returned pair of configurations in \Cref{alg:coupling}. 
  For each $e=(u,v) \in X \oplus Y$, both $u,v \in U$.
\end{proposition}
\begin{proof}
  This directly follows from the coupling procedure.
\end{proof}

Now, we are ready to prove Property (\ref{req:dist-req}).
\begin{proof}[Proof of Property (\ref{req:dist-req})]
Let $U$ be the set of visited vertices upon termination and $(X,Y)$ be the returned configuration. By \Cref{prop:geometric} and \Cref{prop:edge-bound},
\begin{align*}
  \E{\abs{X \oplus Y}} &\le \sum_{k=1}^{\abs{V}} \binom{k}{2} \Pr{\abs{U} = k} \leq \sum_{k=1}^{+\infty} k \Pr{\abs{U} \ge k+1} \le \frac{1}{4\eta_{\min}^2}. \qedhere
\end{align*}
\end{proof}

Now, we only left to prove Property (\ref{req:coupling-req}).
To begin with, we need the following propositions.

\begin{proposition}\label{prop:identity}
  $\mu^{\-{GSW}}_{E,\*p,\*\eta,\as{\sigma}{u}{0}}(\+A) = q_0$  and $\mu^{\-{GSW}}_{E,\*p,\*\eta,\as{\sigma}{u}{1}}(\+A) = q_1$.
\end{proposition}

\begin{proof}
  Without loss of generality, we only prove the first part. It holds that
  \begin{align*}
    \mu^{\-{GSW}}_{E,\*p,\*\eta,\as{\sigma}{u}{0}}(\+A) &= \tp{\sum_{S \in \+A} w^{\-{GSW}}_{E,\*p,\*\eta,\as{\sigma}{u}{0}}(S)}\tp{\sum_{S \in \+A} w^{\-{GSW}}_{E,\*p,\*\eta,\as{\sigma}{u}{0}}(S) + \sum_{S \not\in \+A} w^{\-{GSW}}_{E,\*p,\*\eta,\as{\sigma}{u}{0}}(S)}^{-1}\\
    &= \tp{\eta_u \sum_{S \in \+A} w^{\-{GSW}}_{E,\*p,\as{\*\eta}{u}{1},\sigma}(S)}\tp{\eta_u \sum_{S \in \+A} w^{\-{GSW}}_{E,\*p,\as{\*\eta}{u}{1},\sigma}(S) + \sum_{S \not\in \+A} w^{\-{GSW}}_{E,\*p,\as{\*\eta}{u}{1},\sigma}(S)}^{-1}\\
    &= \frac{\eta_u R}{\eta_u R+1} = q_0.
  \end{align*}
\end{proof}

\begin{fact} \label{prop:cond-dist}
  $\mu^{\-{GSW}}_{E, \*p, \as{\*\eta}{u}{0}, \as{\sigma}{u}{0}} = \mu^{\-{GSW}}_{E, \*p, \*\eta, \sigma}(\cdot \mid \overline{\+A})$ and $\mu^{\-{GSW}}_{E, \*p, \as{\*\eta}{u}{0}, \as{\sigma}{u}{1}} = \mu^{\-{GSW}}_{E, \*p, \*\eta, \sigma}(\cdot \mid \+A)$
  \footnote{We only consider the case where distributions are well-defined, i.e. $Z^{\mathrm{GSW}}_{E, \*p, \as{\*\eta}{u}{c}, \as{\sigma}{u}{c}}>0$ for $c=0,1$.}.
\end{fact}

Now, we are ready to prove Property~\eqref{req:coupling-req}.

\begin{proof}[Proof of Property~\eqref{req:coupling-req}]
  % Before the proof, we first remark that it could be verified that in every recursive call of \Cref{alg:coupling}, the distribution $\mu^{\-{GSW}}_{E, \*p, \*\eta, \sigma}$ and $\mu^{\-{GSW}}_{E, \*p, \*\eta, \*\sigma \oplus \*1_u}$ are both well defined.
  It suffices to prove that, for any valid input $((V,E),\*p,\*\eta,\sigma,u,U)$, a pair of configurations $(X,Y)$ drawn in procedure $\text{Couple}((V,E),\*p,\*\eta,\sigma,u,U)$ satisfies 
    \begin{align*}
      X \sim \mu^{\mathrm{GSW}}_{E,\*p,\*\eta,\sigma} \text{ and } Y \sim \mu^{\mathrm{GSW}}_{E,\*p, \*\eta, \sigma \oplus \*1_u}.
    \end{align*}
  Without loss of generality, we only prove $X \sim \mu^{\mathrm{GSW}}_{E,\*p,\*\eta,\sigma}$. We prove by induction on $m = \abs{E}$.
  
  The base case $m=0$ is trivial.
  Suppose Property~\eqref{req:coupling-req} holds for all $E$ with $\abs{E} = m' < m$.
  We will show that it also holds when $\abs{E} = m$.
  We will considering two cases: (1) $u \in U$; (2) $u \not\in U$.

  When $u \in U$, \Cref{alg:coupling} will 
  \begin{itemize}
  \item select an arbitrary $e = (u,v) \in E_u = \{f \in E \mid f=(u,y) \text{ for some $y$}\}$;
  \item sample $X_1 \sim \nu$, which is the distribution $\mu^{\mathrm{GSW}}_{E,\*p,\*\eta,\sigma}$ projected on $e$;
  \item sample $C_1$ via procedure $\text{Couple}(u,E \setminus \{e\} ,\*\eta,\sigma^\star,U)$ or $\text{Couple}(u,E \setminus \{e\} ,\*\eta,\sigma^\star,U)$, where 
    \begin{align*}
      \sigma^\star = \begin{cases}
                       \sigma &, X_1 = \emptyset;\\
                       \sigma \oplus \*1_u \oplus \*1_v &, X_1 = \{e\}\\
                     \end{cases}
    \end{align*}
  \end{itemize}
  By induction hypothesis, $C_1 \sim \mu^{\mathrm{GSW}}_{E \setminus \{e\} ,\*p,\*\eta,\sigma^\star}$.
  By the definition of $\sigma^\star$, it holds that
  \begin{align*}
    \mu^{\-{GSW}}_{E \setminus \{e\} ,\*p,\*\eta,\sigma^\star}
    = \begin{cases}
        \mu^{\-{GSW}}_{E,\*p,\*\eta,\sigma}\tp{\cdot \mid e} &, X_1 = \{e\},\\
        \mu^{\-{GSW}}_{E,\*p,\*\eta,\sigma}\tp{\cdot \mid \overline{e}} &, X_1 = \emptyset.
      \end{cases}
  \end{align*}
  Hence, $X= C_1 \cup X_1 \sim \mu^{\mathrm{GSW}}_{E,\*p,\*\eta,\sigma}$. 

  When $u \not\in U$, let $q_0,q_1$ be defined in \Cref{line:def} in \Cref{alg:coupling}. \Cref{alg:coupling} behaves as follows:
  %note that \Cref{alg:coupling} remains the same if we add ``Return Couple$(E,\*p,\*\eta,\sigma, U)$'' right after \Cref{line:after-return} in \Cref{alg:coupling}. 
  %Combining previous analysis, it suffices to show the following procedure samples $X \sim \mu^{\mathrm{GSW}}_{E,\*p,\*\eta,\sigma}$:
  \begin{itemize}
  \item with probability $1-q_1$, sample $X$ from distribution $\mu^{\mathrm{GSW}}_{E,\*p,\as{\*\eta}{u}{0},\as{\sigma}{u}{0}}$;\\
  \item with probability $q_0$, sample $X$ from distribution $\mu^{\mathrm{GSW}}_{E,\*p,\as{\*\eta}{u}{0},\as{\sigma}{u}{1}}$;\\
  \item with probability remaining probability $q_1-q_0$, sample $X$ from distribution $\mu^{\mathrm{GSW}}_{E,\*p,\as{\*\eta}{u}{0},\sigma}$.
  \end{itemize}
  Here, the last term follows from the analysis of previous case.
  Hence for $S \subseteq E$, it holds that
  \begin{align*}
    \Pr{X = S} = q_0 \mu^{\-{GSW}}_{E, \*p, \as{\*\eta}{u}{0}, \as{\sigma}{u}{1}}(S) + (q_1 - q_0) \mu^{\-{GSW}}_{E, \*p, \as{\*\eta}{u}{0}, \sigma} (S) + (1 - q_1) \mu^{\-{GSW}}_{E, \*p, \as{\*\eta}{u}{0}, \as{\sigma}{u}{0}} (S).
  \end{align*}
  When $\sigma_u = 0$, it holds that 
  \begin{align*}
    \Pr{X = S}
    &= q_0 \mu^{\-{GSW}}_{E, \*p, \as{\*\eta}{u}{0}, \as{\sigma}{u}{1}}(S) + (1 - q_0) \mu^{\-{GSW}}_{E, \*p, \as{\*\eta}{u}{0}, \as{\sigma}{u}{0}}(S) \\
    &= \mu^{\-{GSW}}_{E, \*p, \*\eta, \as{\sigma}{u}{0}}(\+A) \mu^{\-{GSW}}_{E, \*p, \*\eta, \sigma}(S \mid \+A) + \mu^{\-{GSW}}_{E, \*p, \*\eta, \as{\sigma}{u}{0}}(\overline{\+A}) \mu^{\-{GSW}}_{E, \*p, \*\eta, \sigma}(S \mid \overline{\+A}) \\
   (\text{by } \sigma_u = 0) \quad &= \mu^{\-{GSW}}_{E, \*p, \*\eta, \sigma} (S),
  \end{align*}
  where the second equation holds by \Cref{prop:identity}, \Cref{prop:cond-dist}.

  Similarly, when $\sigma_u = 1$, it holds that
  \begin{align*}
    \Pr{X = S}
    &= q_1 \mu^{\-{GSW}}_{E, \*p, \as{\*\eta}{u}{0}, \as{\sigma}{u}{1}}(S) + (1 - q_1) \mu^{\-{GSW}}_{E, \*p, \as{\*\eta}{u}{0}, \as{\sigma}{u}{0}}(S) \\
    &= \mu^{\-{GSW}}_{E, \*p, \*\eta, \as{\sigma}{u}{1}} (\+A) \mu^{\-{GSW}}_{E,\*p, \*\eta, \sigma}(S \mid \+A) + \mu^{\-{GSW}}_{E, \*p, \*\eta, \as{\sigma}{u}{1}} (\overline{\+A}) \mu^{\-{GSW}}_{E, \*p, \*\eta, \sigma}(S \mid \overline{\+A}) \\
   (\text{by } \sigma_u = 1) \quad &= \mu^{\-{GSW}}_{E, \*p, \*\eta, \sigma} (S),
  \end{align*}
  where the second equation holds by \Cref{prop:identity}, \Cref{prop:cond-dist}.

  Combining these two cases, we have $X \sim \mu^{\-{GSW}}_{E, \*p, \*\eta, \sigma}$ and hence complete the proof.
 % We now discuss the following cases separately.
 % \begin{itemize}
 % \item When $S \in \+A$ and $\sigma_u = 0$, it holds that
 %   \begin{align*}
 %     \Pr{X = S} = q_0 \mu^{\-{GSW}}_{E,\*p,\tilde{\*\eta},\sigma \cup \*1_u}(S) = q_0 \mu^{\-{GSW}}_{E,\*p,\*\eta,\sigma}\tp{S \mid \+A} \overset{(*)}{=}  \mu^{\-{GSW}}_{E,\*p,\*\eta,\sigma}\tp{S}.
 %   \end{align*}  
 % \item When $S \not \in \+A$ and $\sigma_u = 0$, it holds that
 %   \begin{align*}
 %     \Pr{X = S} = (1-q_1+q_1-q_0) \mu^{\-{GSW}}_{E,\*p,\tilde{\*\eta},\sigma}(S) = (1-q_0) \mu^{\-{GSW}}_{E,\*p,\*\eta,\sigma}\tp{S \mid \overline{\+A}} \overset{(*)}{=}  \mu^{\-{GSW}}_{E,\*p,\*\eta,\sigma}\tp{S}.
 %   \end{align*}
 % \item When $S \in \+A$ and $\sigma_u = 1$, it holds that
 %   \begin{align*}
 %     \Pr{X = S} = (q_0 + q_1 - q_0) \mu^{\-{GSW}}_{E,\*p,\tilde{\*\eta},\sigma}(S) = q_1 \mu^{\-{GSW}}_{E,\*p,\*\eta,\sigma}\tp{S \mid \+A} \overset{(\star)}{=}  \mu^{\-{GSW}}_{E,\*p,\*\eta,\sigma}\tp{S}.
 %   \end{align*}  
 % \item When $S \not\in \+A$ and $\sigma_u = 1$, it holds that
 %   \begin{align*}
 %     \Pr{X = S} = (1-q_1) \mu^{\-{GSW}}_{E,\*p,\tilde{\*\eta},\sigma \cap (\*1 \oplus \*1_u)}(S) = (1-q_0) \mu^{\-{GSW}}_{E,\*p,\*\eta,\sigma}\tp{S \mid \overline{\+A}} \overset{(\star)}{=}  \mu^{\-{GSW}}_{E,\*p,\*\eta,\sigma}\tp{S}.
 %   \end{align*}
 % \end{itemize}
 % Here, $(*)$ and $(\star)$ follows from the first part and second part of \Cref{prop:identity}. This concludes the proof of requirement~\eqref{req:coupling-req}.
\end{proof}

\subsection{Lifting coupling independence} \label{sec:lift}
In this section, we will prove \Cref{lm:lifting-no-pin}.
Let $\mu$ be the distribution of random cluster model specified by graph $G = (V, E)$, and parameters $\*p \in [0, 1]^E$, $\*\lambda \in [0, 1)^V$.
Furthermore, let $\nu$ be the distribution of the subgraph-world model model specified by the same graph $G$, and parameters $\frac{\*p}{2} = \tp{\frac{p_e}{2}}_{e\in E}$ and $\eta = \tp{\frac{1 - \lambda_v}{1 + \lambda_v}}_{v \in V}$.
A natural coupling between $\mu$ and $\nu$ is observed by previous works~\cite{guo2018random, feng2022sampling}, which is summarized below.
\begin{lemma}[\text{\cite[Lemma 3.3]{feng2022sampling}}] \label{lem:couple-mu-nu}
  Let $\*q := \tp{p_i / (2 - p_i)}_{i \in E}$.
  Suppose $Z \sim \bigotimes_{i \in E} \-{Ber}(q_i)$, $X \sim \nu$, and $Y = X \cup Z$ then it holds that $Y \sim \mu$.
  Equivalently, for every $Y \subseteq E$, we have
  \begin{align} \label{eq:couple-mu-nu}
    \mu(Y) = \sum_{X\subseteq Y} \nu(X) \prod_{h \in Y\setminus X} q_h \prod_{f \in E\setminus Y} (1 - q_f).
  \end{align}
\end{lemma}

Now, fix $e \in E$, let $(X_0, X_1)$ be a coupling of $\nu(\cdot \mid \overline{e})$ and $\nu(\cdot \mid e)$.
Let $t_e := \frac{q_e \nu(\overline{e})}{q_e\nu(\overline{e}) + \nu(e)}$ be a real number.
We will construct $(Y_0, Y_1)$ as a coupling of $\mu(\cdot \mid \overline{e})$ and $\mu(\cdot \mid e)$ as follow:
\begin{itemize}
\item sample $Z \sim \bigotimes_{f \in E}\-{Ber}\tp{q_f}$;
\item with probability $t_e$, let $Y_0 = (X_0 \cup Z) \setminus \{e\}$ and $Y_1 = (X_0 \cup Z) \cup \{e\}$;
\item with probability $1 - t_e$, let $Y_0 = (X_0 \cup Z) \setminus \{e\}$ and $Y_1 = X_1 \cup Z$.
\end{itemize}

Now, \Cref{lm:lifting-no-pin} could be simply proved by
\begin{align*}
  \E{\abs{Y_0 \oplus Y_1}}
  &= t_e + (1 - t_e) \cdot \E{\abs{(X_0 \cup Z)\setminus \{e\} \oplus X_1 \cup Z}} \\
  &\leq t_e + (1 - t_e) \cdot \E{\abs{X_0 \oplus X_1}} \leq C,
\end{align*}
where in the last inequality, we use the fact that $C \geq 1$, which holds by definition.
 
Now, we only need to verify that $Y_0, Y_1$ follow the correct distribution as we claimed.
For convenience, for $S \subseteq E$, we use $\*q^S$ to denote $\prod_{i \in S} q_i$.
Then for any $Y \subseteq E \setminus \{e\}$, 
\begin{align*}
  \Pr{Y_0 = Y}
  &= \sum_{T \subseteq Y} \nu(T \mid \overline{e}) \*q^{Y\setminus T} (1 - \*q)^{E\setminus Y \setminus \{e\}} \\
  &= \frac{\sum_{T \subseteq Y} \nu(T) \*q^{Y\setminus T} (1 - \*q)^{E\setminus Y}}{\nu(\overline{e}) (1 - q_e)}
   \overset{(\star)}{=} \frac{\mu(Y)}{\mu(\overline{e})} = \mu(Y \mid \overline{e}),
\end{align*}
where in $(\star)$ we use \eqref{eq:couple-mu-nu} and the fact that $\mu(\overline{e}) = \nu(\overline{e}) (1 - q_e)$, which could be implied from \Cref{lem:couple-mu-nu}.
A similar calculation shows that for any $\{e\} \subseteq Y \subseteq E$,
\begin{align*}
  \Pr{Y_1 = Y}
  &= t_e \sum_{T \subseteq Y \setminus\{e\}} \nu(T \mid \overline{e}) \*q^{Y\setminus\{e\}\setminus T} (1 - \*q)^{E\setminus Y} + (1 - t_e) \sum_{\{e\} \subseteq T \subseteq Y} \nu(T\mid e) \*q^{Y\setminus T} (1 - \*q)^{E\setminus Y} \\
  &= \tp{\textstyle t_e \sum_{T\subseteq Y\setminus \{e\}}\frac{\nu(T)}{\nu(\overline{e})} \*q^{Y\setminus\{e\}\setminus T} + (1 - t_e) \sum_{\{e\} \subseteq T \subseteq Y} \frac{\nu(T)}{\nu(e)} \*q^{Y\setminus T}} (1 - \*q)^{E\setminus Y} \\
  &\overset{(+)}{=} \frac{\tp{\sum_{T\subseteq Y\setminus\{e\}} \nu(T) \*q^{Y\setminus T} + \sum_{\{e\} \subseteq T \subseteq Y} \nu(T) \*q^{Y\setminus T}}(1 - \*q)^{E\setminus Y}}{q_e \nu(\overline{e})  + \nu(e)} \\
  &= \frac{\sum_{T\subseteq Y} \nu(T) \*q^{Y\setminus T}(1 - \*q)^{E\setminus Y}}{q_e \nu(\overline{e})  + \nu(e)}
   \overset{(\star)}{=} \frac{\mu(Y)}{\mu(e)} = \mu(Y \mid e),
\end{align*}
where $(+)$ holds by the fact that $t_e = \frac{q_e \nu(\overline{e})}{q_e \nu(\overline{e}) + \nu(e)}$ and $(\star)$ holds by \Cref{lem:couple-mu-nu}.

\section{Rapid mixing of Glauber dynamics in good regime}\label{sec:good-regime}
Let $\mu$ be the distribution of the random cluster model specified by graph $G=(V,E)$, parameters $\*p \in (0,1)^E$ and $\*\lambda \in (0,1)^V$. 
Let $P^{\-{GD}}_\mu$ be the Glauber dynamics with stationary distribution $\mu$.
%Furthermore, we assume $n = \abs{V} \ge 10$, $m = \abs{V} \ge 10$ and  {\color{red} $K = (1 - p_{\min}) \log n$}.

% Consider two Markov chains $\{X_t\}_{t \geq 0}$ {\color{red} starting from $X_0 \in \Omega$} and $\{Y_t\}_{t \geq 0}$ defined as 
% \begin{itemize}
% \item Initialize $Y_0 \sim \mu$.
% \item In the $i$-th step, sample $X_i \sim P^{\-{GD}}_\mu(X_{i-1}, \cdot)$ and $Y_i \sim P^{\-{GD}}_\mu(Y_{i-1}, \cdot)$.
% \end{itemize}

We restate \Cref{lem:GD} for convenience.

% \begin{lemma}\label{thm:good-mixing}
%   Let $\epsilon \in (0,1)$ be a real number.
%   If
%   \begin{align} \label{eq:good-mix-cond}
%     K \leq \min\left\{10^{-5} \exp\tp{-\frac{\log (8/\epsilon)}{\log n}}, \frac{1 - \lambda_{\max}}{27} \right\}
%     \text{ and }
%     T \ge 2m(\log m + \log(2/\epsilon)),
%   \end{align}
%   then there exists a coupling of $\{X_t\}_{t\geq 0}$ and $\{Y_t\}_{\geq 0}$ such that
%   \begin{align*}
%     \Pr{X_T \neq Y_T} &\leq \epsilon.
%   \end{align*} 
% \end{lemma}

\begin{lemma}\label{thm:good-mixing}
  Let $\epsilon \in \tp{0,\frac{1}{4}}$ be a real number.
  If
  % {\color{red}
  \begin{align}\label{eq:good-mix-cond}
    K := \tp{1-p_{\min}}\log n \leq \min\left\{10^{-7}, \frac{1-\lambda_{\max}}{27} \right\},
  \end{align}
  then the mixing time of Glauber dynamics satisfies
  \begin{align*}
    T_{\mathrm{mix}}(\epsilon) \le 25 m \log m \log (1/\epsilon) + 1.
  \end{align*}
  Furthermore, if
  \begin{align*}
    K \leq \min\left\{10^{-5} \exp\tp{-\frac{\log (8/\epsilon)}{\log n}}, \frac{1 - \lambda_{\max}}{27} \right\},
  \end{align*}
  then the mixing time of Glauber dynamics initialized from $E$ satisfies
  \begin{align*}
    T_{\mathrm{mix}}(\epsilon,E) \le 2m(\log m + \log(2/\epsilon)).
  \end{align*} 
\end{lemma}

We remark that a trade-off exists between the mixing time and the condition on $K$. 
This trade-off arises from a subtle distinction in the proof of the mixing time. 
To prove the first part, we employ the burn-in method and a ``doubling'' argument, whereas the second part does not rely on these techniques. 
We will use the first part of~\Cref{thm:good-mixing} in the subsequent analysis, and the proof of the second part is deferred to~\Cref{append:proof}.

% Let $\alpha > 0$ be a parameter that we will tune later.
%In this chapter, we assume that $p_{\min} = \min_{e \in E} p_e \geq 1-K/\log n$ with $K<\frac{1-\lambda_{\max}}{27}$. 
In the rest of this section, we will always assume that $K$ is picked as in \eqref{eq:good-mix-cond}.
%Let $n = \abs{V}$ be the number of vertices and $m = \abs{E\setminus\Lambda}$ be the number of edges without pinning.
% Let $\+C$ be the set of vertex subsets with $\alpha \log n$ edge expansion in the graph $(V, E\setminus \Lambda)$, that is
Define $\+C$ by
\begin{align*}
  \+C := \{ S \subseteq V \mid \abs{S} \leq n/2 \text{ and } \abs{E(S, V\setminus S)} \geq \abs{S} \log n\},
\end{align*}
where $E(S,V \setminus S)$ is the set of edges between $S$ and $V \setminus S$.
Furthermore, define the good event $\+G$  by
\begin{align*}
  \+G := \{ X \subseteq E \mid \forall S \in \+C, \abs{X \cap E(S,V\setminus S)} > 0\}.
\end{align*}

Intuitively, the good event $\+G$ is the set of configurations $X \subseteq E$, such that there exists no connected components $C$ in graph $(V,X)$ with large $\abs{E(C, V\setminus C)}$.

Let $(X_t)_{t \ge 0}$ and $(Y_t)_{t \ge 0}$ be Glauber dynamics starting from different configurations $X_0$ and $Y_0$. 
We will show that after a burn-in phase of length $O(m \log m)$, the good event $\+G$ happens in high probability.
Furthermore, when $X_t, Y_t \in \+G$, there exists a coupling so that the hamming distance of $X_t$ and $Y_t$ contracts with rate $1-\frac{1}{2m}$, where $m=\abs{E}$.
% We have the following results.

\begin{lemma}\label{lem:bad-event-bound}
  If $t \ge \theta m \log m$ for some $\theta > 0$, it holds that 
  \begin{align*}
  \Pr{X_t \not\in \+G} \leq n^{\log (27 K)} + m^{1-\theta} \text{ and } \Pr{Y_t \not\in \+G} \leq n^{\log (27 K)} + m^{1-\theta}.
  \end{align*}
\end{lemma}
\begin{lemma}\label{lem:pc-good}
  % Suppose we have $X_t \in \+A(\alpha)$, $Y_t \in \+A(\alpha)$.
  % $p \geq 1 - K/\log n$ for some parameter $K > 0$ satisfying,
  % \begin{align} \label{eq:pc-cond}
  %   K \leq \frac{1}{20} \cdot  \frac{1 - \lambda_{\max}}{\lambda_{\max}}\log n
  %   \quad \text{and} \quad
  %   \alpha K \leq \frac{1}{20} \cdot \frac{\log (1/\lambda_{\max})}{\lambda_{\max}^{1/\log(1/\lambda_{\max})}},
  % \end{align}
  % then we have
  % \begin{align*}
  %   \E{\abs{X_{t+1} \oplus Y_{t+1}} \mid X_t, Y_t} \leq \tp{1 - \frac{1}{2m}} \abs{X_t \oplus Y_t}.
  % \end{align*}
  When $K \leq \frac{1 - \lambda_{\max}}{27}$, there exists a coupling of $(X_t)_{t \ge 0}$ and $(Y_t)_{t \ge 0}$ so that when $X_t,Y_t \in \+G$,
  \begin{align*}
    \E{\abs{X_{t+1} \oplus Y_{t+1}} \mid X_t, Y_t} \leq \tp{1 - \frac{1}{2m}} \abs{X_t \oplus Y_t}.
  \end{align*}
\end{lemma}
The proof of \Cref{lem:bad-event-bound} and \Cref{lem:pc-good} is deferred to \Cref{sec:bad-event-bound} and \Cref{sec:pc-good} respectively.
To prove \Cref{thm:good-mixing} via \Cref{lem:bad-event-bound} and \Cref{lem:pc-good}, we need the coupling with stationary lemma.

\begin{lemma}[\text{\cite[Theorem 3.1]{hayes2006coupling}}]\label{lem:cfs}
  Let $(X_t)_{t \ge 0}$ and $(Y_t)_{t \ge 0}$ be coupled Markov chains.
  Let $\epsilon \in (0,1)$, $T_0, T_1 \in \mathbb{N}$ be parameters satisfying $T_0 < T_1$. 
  Furthermore, $\+E_t$ denote the event 
  \begin{align*}
    \E{\abs{X_{t+1} \oplus Y_{t+1}} \mid X_t, Y_t} \leq \tp{1 - \frac{1}{2m}} \abs{X_t \oplus Y_t}.
  \end{align*}
  Suppose $\Pr{\overline{ \+E_t}} \leq \delta$ for all $T_0 \le t < T_1$, then 
  \begin{align*}
    \Pr{X_{T_1} \neq Y_{T_1}} \leq \tp{1 - \frac{1}{2m}}^{T_1 - T_0}m + \delta \cdot 2m^2.
  \end{align*}
\end{lemma}
% To be more compatible with our context, \Cref{lem:cfs} is stated as a specialized variant of its origianl version.
We are now ready to prove \Cref{thm:good-mixing}.
\begin{proof}[Proof for the first part of~\Cref{thm:good-mixing}]
  Without loss of generality, we will only prove the case where $\epsilon = \frac{1}{4}$. 
  This can be extended to general case by a doubling argument~\cite{levin2017markov}.
  Moreover, we assume $n \geq 2$ and $m \geq 2$.
  Since $n \leq 1$ implies that $m = 0$ and when $m \leq 1$, Glauber dynamics mixes in one step.
  
  Let $T_0 = 9m \log m$, $T_1 = 15 m \log m$ and $X_0,Y_0$ be arbitrary starting configurations. 
  We consider the Glauber dynamics $(X_t)_{t \ge 0}$ and $(Y_t)_{t \ge 0}$ starting from $X_0$ and $Y_0$ respectively.
  By \Cref{lem:bad-event-bound} and \Cref{lem:pc-good}, there exists a coupling of $(X_t)_{t \ge 0}$ and $(Y_t)_{t \ge 0}$ such that for all $T_0 \le t < T_1$, the event $\+E_t$ happens with probability no more than $2 n^{\log (27 K)} + 2 m^{-8}$, where $\+E_t$ denotes the event $\mathbb{E}\left[\abs{X_{t+1} \oplus Y_{t+1}} \mid X_t, Y_t \right] \le \tp{1-\frac{1}{2m}} \abs{X_t \oplus Y_t}$.
 % and $\+E$ be the event 
 % \begin{align*}
 %   \E{\abs{X_{t+1} \oplus Y_{t+1}} \mid X_t, Y_t} \leq \tp{1 - \frac{1}{2m}} \abs{X_t \oplus Y_t}, \quad \forall 0 \le t < T.
 % \end{align*}
  % \begin{align*}
  %   \forall 0\leq t \leq T - 1, \quad \Pr{\overline{\+E_t}} &\leq 2 n^{\log(27K)} + 2m^{-4}.
  % \end{align*}
  %\begin{align*}
  %  \Pr{\+E} \ge 1 - 2Tn^{\log (27K)} = 1 - 4m^2 n^{\log (27K)} - 4m\log \tp{\frac{2}{\epsilon}} \ge 1-n^{\log(10^5 K)} - n^{\log(10^5 K)} \log \tp{\frac{2}{\epsilon}}. 
  %\end{align*}
  By $K \leq 10^{-7}$, it holds that
  \begin{align*}
    \tp{2n^{\log(27K)} + 2m^{-8}} \cdot 2m^2 \leq 4 n^{4 + \log(27K)} + 4m^{-6} \leq 4n^{-6} +4 m^{-6} \le \frac{1}{8},
  \end{align*}
  where the last inequality follows from our assumption $n,m \ge 2$.
  %\begin{align*}
  %  n^{\log (10^5 K)} \log \tp{\frac{2}{\epsilon}} \le \frac{n^{-\frac{2 \log \epsilon}{\log n}}}{4 \epsilon} = \frac{\epsilon}{4}.    
  %\end{align*}
  %Therefore, 
  %\begin{align*}
  %  \Pr{\+E} \ge 1-\frac{\epsilon}{2}.
  %\end{align*}
  Therefore, by \Cref{lem:cfs},
  \begin{align*}
    \Pr{X_{T_1} \neq Y_{T_1}} &\le \tp{1-\frac{1}{2m}}^
{T_1-T_0} m + \frac{1}{8} \leq \frac{1}{4}.
 \end{align*}
 Finally, by the geometric convergence of Markov chain, it holds that
  \begin{align*}
    T_{\mathrm{mix}}(\epsilon)
    &\le 15 m \log m \log_2 (1/\epsilon)
    \le 25 m \log m \log (1/\epsilon) + 1,
  \end{align*}
  where the $+1$ is used to handle the cases where $n \leq 1$ or $m \leq 1$.
\end{proof}

\subsection{Contraction in $\+G$ (proof of \Cref{lem:pc-good})}\label{sec:pc-good}
% \todo{simplify the notation}
For any $S, T \in \+G$, let $e_1,e_2,\ldots,e_x$ be edges in $S \setminus T$, and $f_1,f_2,\ldots,f_y$ be edges in $T \setminus S$.
We design the following path of configurations $P=(P_0,P_1,\ldots,P_{x+y})$ from $S$ to $T$.
\begin{align*}
  \forall 0 \le i \le s+t, \quad P_i = 
  \begin{cases}
    S & \text{if }i=0,\\
    P_{i-1} \cup \{f_i\} & \text{if }1 \le i \le y,\\
    P_{i-1} \setminus \{e_{i-y}\} & \text{if }i > y.
  \end{cases}
\end{align*}
% Start from $S = A$, we could first add all the elements from $B\setminus A$ to $S$ one by one, then remove all the elements in $A\setminus B$ from $S$ one by one, finally we get $S = B$.
Note that $\abs{P_{i-1} \oplus P_i} = 1$ for all $1 \le i \le x+y$, $P_i \in \+G$ for all $0 \le i \le x+y$, and the length of this path is exactly $\abs{S\oplus T}$. 
Hence, by the standard path coupling argument~\cite{bubley1997path}, 
it suffices to prove that, for every $X_t, Y_t \in \+G$ satisfying $\abs{X_t \oplus Y_t} = 1$,
\begin{align} \label{eq:pc-target}
  \E{\abs{X_{t+1} \oplus Y_{t+1}} \mid X_t, Y_t} &\leq 1 - \frac{1}{2m}.
\end{align}

Without loss of generality, we assume that $Y_t = X_t \cup \{e\}$, where $e = (u,v) \in E$.
Using the one step optimal coupling of the Glauber dynamics, it holds that
\begin{align*}
  \E{\abs{X_{t+1} \oplus Y_{t+1}} \mid X_t, Y_t} = 1 - \frac{1}{m} + \frac{1}{m}\sum_{f \in E \setminus \{e\}} d(X_t,Y_t,f),
\end{align*}
where $d(X_t,Y_t,f)=\abs{\frac{\mu(X_t \cup \{f\})}{\mu(X_t \cup \{f\}) + \mu(X_t \setminus \{f\})} - \frac{\mu(Y_t \cup \{f\})}{\mu(Y_t \cup \{f\}) + \mu(Y_t \setminus \{f\})}}$.

% where $m := \abs{E\setminus \Lambda}$ denotes the set of unpinned edges and $\sigma := X_t(E\setminus\Lambda\setminus\{e, f\})$ be the pinning of all edges except $e$ and $f$.

%  $C_v(M)$ the connected component that contains vertex $v$ in graph $(V, M)$ for any $M \subseteq E$.
% Also, let $\kappa(M)$ denote the set of all the connected components in graph $(G, M)$.

We first discuss the value of $d(X_t,Y_t,f)$. 
Fix edge $f= (x,y) \neq e$, let 
\begin{align*}
  S_{\overline{f}, \overline{e}} := X_t \setminus \{f\}, S_{f, \overline{e}} := X_t \cup\{f\}, S_{\overline{f}, e} := Y_t \setminus \{f\}, S_{f, e} := Y_t \cup\{f\}.
\end{align*}

It can be verified that
\begin{align*}
  \frac{\mu(X_t \cup \{f\})}{\mu(X_t \cup \{f\}) + \mu(X_t \setminus \{f\})} = 
  \begin{cases}
    p_f, & \text{if }C_{x, \overline{f},\overline{e}} = C_{y, \overline{f},\overline{e}}\;,\\
    \frac{1+\*\lambda^{C_{x, \overline{f},\overline{e}} \cup C_{y, \overline{f},\overline{e}}}}{1+\*\lambda^{C_{x, \overline{f},\overline{e}} \cup C_{y, \overline{f},\overline{e}}} + (1-p_f) \tp{\*\lambda^{C_{x, \overline{f},\overline{e}}} + \*\lambda^{C_{y, \overline{f},\overline{e}}}}} & \text{otherwise},
  \end{cases}
\end{align*}

\begin{align*}
  \frac{\mu(Y_t \cup \{f\})}{\mu(Y_t \cup \{f\}) + \mu(Y_t \setminus \{f\})} = 
  \begin{cases}
    p_f & \text{if }C_{x, \overline{f},e} = C_{y, \overline{f},e}\;,\\
    \frac{1+\*\lambda^{C_{x, \overline{f},e} \cup C_{y, \overline{f},e}}}{1+\*\lambda^{C_{x, \overline{f},e} \cup C_{y, \overline{f},e}} + (1-p_f) \tp{\*\lambda^{C_{x, \overline{f},e}} + \*\lambda^{C_{y, \overline{f},e}}}} & \text{otherwise},
  \end{cases}
\end{align*}
where $C_{r,i,j}$ is the connected component containing $r$ in graph $(V,S_{i,j})$ where $i \in \{f,\overline{f}\}$ and $j \in \{e,\overline{e}\}$, and $\*\lambda^S = \prod_{r \in S} \lambda_r$.
For general subset $R \subseteq E$, we also use $C_r(R)$ to denote the connected component containing $r$ in graph $(V, R)$.
% For convenience, we use $C_u(f), C_v(f)$ to denote $C_u(S_{\overline{f}, \overline{e}}), C_{v, \overline{f},\overline{e}}$, respectively, and when the context is clear, we omit $f$.
% Note that $\nu(f \mid e, \sigma)$ is determined by the following ratio
% \begin{align} \label{eq:e+}
%   \frac{\nu(f \mid e, \sigma)}{\nu(\overline{f} \mid e, \sigma)} &= \frac{\nu(f, e, \sigma)}{\nu(\overline{f}, e, \sigma)} = \frac{p_f}{1 - p_f} \cdot \frac{\prod_{A \in \kappa(S_{f, e})} (1 + \lambda^A)}{\prod_{B \in \kappa(S_{\overline{f}, e})}(1 + \lambda^B)},
% \end{align}
% and similarly, $\nu(f \mid \overline{e}, \sigma)$ is determined by
% \begin{align} \label{eq:e-}
%   \frac{\nu(f \mid \overline{e}, \sigma)}{\nu(\overline{f} \mid \overline{e}, \sigma)} &= \frac{\nu(f, \overline{e}, \sigma)}{\nu(\overline{f}, \overline{e}, \sigma)} = \frac{p_f}{1 - p_f} \cdot \frac{\prod_{A \in \kappa(S_{f, \overline{e}})} (1 + \lambda^A)}{\prod_{B \in \kappa(S_{\overline{f}, \overline{e}})}(1 + \lambda^B)}.
% \end{align}
% Hence, to determine the contribution of $f$, we only need to consider the structure of connected components in these four graphs.
% \todo{add reasoning on classification}
We consider the following cases as illustrated in \Cref{fig:case2-3}.
\begin{enumerate}
\item $C_{u, \overline{f}, \overline{e}} = C_{v, \overline{f},\overline{e}}$;
\item $C_{u,\overline{f}, \overline{e}} \neq C_{v, \overline{f},\overline{e}}$ and $C_{u,f, \overline{e}} = C_{v,f, \overline{e}}$;
\item $C_{u,\overline{f}, \overline{e}} \neq C_{v, \overline{f},\overline{e}}$ and $C_{u,f, \overline{e}} \neq C_{v,f, \overline{e}}$.
\end{enumerate}

  \begin{figure}
    \begin{tikzpicture}[gnode/.style={circle, draw=green!60, fill=green!30, very thick, minimum size=2mm, inner sep=0pt},
      rnode/.style={circle, draw=red!60, fill=red!30, very thick, minimum size=2mm, inner sep=0pt},
      bnode/.style={circle, draw=blue!60, fill=blue!30, very thick, minimum size=2mm,inner sep=0pt},
      cnode/.style={circle, draw=cyan!5, fill=cyan!5, very thick, minimum size=2mm,inner sep=0pt},
      rect/.style={rectangle, draw=red!40, fill=red!5, ultra thick, rounded corners, minimum width=40mm, minimum height=30mm},
      arect/.style={rectangle, draw=red!40, fill=red!5, ultra thick, rounded corners, minimum width=60mm, minimum height=30mm},
      brect/.style={rectangle, draw=green!40, fill=green!5, ultra thick, rounded corners, minimum width=20mm, minimum height=30mm}
      ] 
      \node[brect] (rec1) at (-3.9,3){};
      \draw[draw=cyan!40, fill=cyan!5, very thick] (-3.9,3) ellipse (0.7 and 1.2);
      \node[bnode] (n1) at (-3.9,2.5){};
      \node[xshift=-8pt] at (n1.south) {\tiny $v$}; 
      \node[bnode] (n2) at (-3.9,3.5){};
      \node[xshift=-8pt] at (n2.south) {\tiny $u$}; 
      \path[dotted, very thick] (n1) edge node[xshift=5pt] {\tiny $e$} (n2);

      \node[rect] (rec1) at (0,3){};
      \draw[draw=cyan!40, fill=cyan!5, very thick] (-0.9,3) ellipse (0.7 and 1.2);
      \draw[draw=cyan!40, fill=cyan!5, very thick] (0.9,3) ellipse (0.7 and 1.2);
      \node[bnode] (n1) at (-0.5,2.7){};
      \node[yshift=-5pt] at (n1.south) {\tiny $x$}; 
      \node[bnode] (n2) at (0.5,2.7){};
      \node[yshift=-5pt] at (n2.south) {\tiny $y$}; 
      \path[dotted, very thick] (n1) edge node[yshift=5pt] {\tiny $f$} (n2);
      \node[bnode] (n1) at (-0.5,3.3){};
      \node[yshift=-5pt] at (n1.south) {\tiny $u$}; 
      \node[bnode] (n2) at (0.5,3.3){};
      \node[yshift=-5pt] at (n2.south) {\tiny $v$}; 
      \path[dotted, very thick] (n1) edge node[yshift=5pt] {\tiny $e$} (n2);

      \node[arect] (rec1) at (5.9,3){};
      \draw[draw=cyan!40, fill=cyan!5, very thick] (4.1,3) ellipse (0.7 and 1.2);
      \draw[draw=cyan!40, fill=cyan!5, very thick] (5.9,3) ellipse (0.7 and 1.2);
      \draw[draw=cyan!40, fill=cyan!5, very thick] (7.7,3) ellipse (0.7 and 1.2);
      \node[bnode] (n1) at (4.5,3){};
      \node[yshift=-5pt] at (n1.south) {\tiny $x$}; 
      \node[bnode] (n2) at (5.5,3){};
      \node[yshift=-5pt] at (n2.south) {\tiny $y$}; 
      \path[dotted, very thick] (n1) edge node[yshift=5pt] {\tiny $f$} (n2);
      \node[bnode] (n1) at (6.3,3){};
      \node[yshift=-5pt] at (n1.south) {\tiny $u$}; 
      \node[bnode] (n2) at (7.3,3){};
      \node[yshift=-5pt] at (n2.south) {\tiny $v$}; 
      \path[dotted, very thick] (n1) edge node[yshift=5pt] {\tiny $e$} (n2);
    \end{tikzpicture}
    \caption{Illustration of three cases. Each ellipse is a connected component in graph $(V,S_{\overline{f},\overline{e}})$.}\label{fig:case2-3}
    \end{figure}
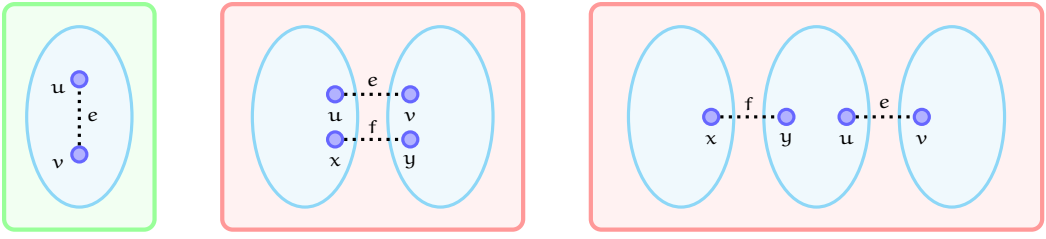

\subsubsection{Case (1)}
% Since $C_{u, \overline{f}, \overline{e}}} = C_{v, \overline{f},\overline{e}}$ and $C_u(S_{f, \overline{e}}) = C_v(S_{f, \overline{e}})$, we know that $\kappa(S_{\overline{f}, \overline{e}}) = \kappa(S_{\overline{f}, e})$ and $\kappa(S_{f, \overline{e}}) = \kappa(S_{f, e})$.
% So, it holds that \eqref{eq:e+} $=$ \eqref{eq:e-}, and we have $\nu(f \mid e, \sigma) = \nu(f \mid \overline{e}, \sigma)$.
In this case, $(V, S_{\overline{f}, \overline{e}})$ and $(V, S_{\overline{f}, e})$ have the same structure of connected components.
Hence we have, $C_{x, \overline{f},\overline{e}} = C_{x, \overline{f},e}$ and $C_{y, \overline{f},\overline{e}} = C_{y, \overline{f},e}$, which implies $d(X_t,Y_t,f) = 0$.

\subsubsection{Case (2)}
In this case, both edges $e=(u,v)$ and $f=(x,y)$ connect $C_{u, \overline{f},\overline{e}}$ and $C_{v, \overline{f},\overline{e}}$ in graph $(V,S_{\overline{f},\overline{e}})$.
Thus, $C_{x, \overline{f},\overline{e}} \neq C_{y, \overline{f},\overline{e}}$ and $C_{x, \overline{f},e} = C_{y, \overline{f},e}$, which implies

% Note that we have $C_{u, \overline{f}, \overline{e}}} \neq C_{v, \overline{f},\overline{e}}$ and $C_u(S_{f, \overline{e}}) = C_v(S_{f, \overline{e}})$ in this case.
% And we also have $C_u(S_{\overline{f}, e}) = C_v(S_{\overline{f}, e})$ and $C_u(S_{f, e}) = C_v(S_{f, e})$.
% %So, in this case, it is likely that \eqref{eq:e+} $\neq$ \eqref{eq:e-}.
% %For convenience, let $f = \{x, y\}$, and $x \in C_u(S_{\overline{f}, \overline{e}}), y \in C_{v, \overline{f},\overline{e}}$.
% Note that $\kappa(S_{f, e}) = \kappa(S_{\overline{f}, e})$, and we have $\nu(f \mid e, \sigma) = p_f$.
% On the other hand, note that $u, v$ are in different connected components $C_u, C_v$ in $\kappa(S_{\overline{f}, \overline{e}})$, but they are merged to one connected component $C_u \uplus C_v$ in $\kappa(S_{f, \overline{e}})$, this means
% \begin{align*}
%   \nu(f \mid \overline{e}, \sigma) &= \frac{p_f (1 + \lambda^{C_u \uplus C_v})}{p_f (1 + \lambda^{C_u \uplus C_v}) + (1 - p_f) (1 + \lambda^{C_u})(1 + \lambda^{C_v})}.
% \end{align*}
% This means
\begin{align}\label{eq:pc-bd-case2} 
  % \abs{\nu(f \mid e, \sigma) - \nu(f \mid \overline{e}, \sigma)}
  d(X_t,Y_t,f)
  % \nonumber &= \abs{p_f - \frac{p_f (1 + \lambda^{C_u \uplus C_v})}{p_f (1 + \lambda^{C_u \uplus C_v}) + (1 - p_f) (1 + \lambda^{C_u})(1 + \lambda^{C_v})}} \\ 
  \nonumber &= \frac{p_f (1 - p_f) (\*\lambda^{ C_{u, \overline{f},\overline{e}}} + \*\lambda^{ C_{v, \overline{f},\overline{e}}})}{1 + (1 - p_f) (\*\lambda^{ C_{u, \overline{f},\overline{e}}} + \*\lambda^{ C_{v, \overline{f},\overline{e}}}) + \*\lambda^{ C_{u, \overline{f},\overline{e}} \cup  C_{v, \overline{f},\overline{e}}}}  \\
  % \nonumber &\leq p_f (1 - p_f) (\lambda^{C_u} + \lambda^{C_v}) \\
  &\leq (1 - p_{\min}) (\lambda_{\max}^{\abs{ C_{u, \overline{f},\overline{e}}}} + \lambda_{\max}^{\abs{ C_{v, \overline{f},\overline{e}}}}).
\end{align}

We further consider two sub-cases: (a) $f \in X_t$; (b) $f \not\in X_t$.

\paragraph{\textbf{Case (2a)}}
As $f \in X_t$, it holds that $C_u(X_t) = C_v(X_t)$ and there exists a path $P = (x_0,x_1,\ldots,x_\ell)$ that connects $u$ and $v$ in the graph $(G, X_t)$ for some $\ell > 0$.
Note that only edges on path $P$ may satisfy the requirement of Case (2) i.e. $C_{u, \overline{f},\overline{e}} \neq C_{v, \overline{f},\overline{e}}$.
Therefore, by \eqref{eq:pc-bd-case2}, the sum of $d(X_t,Y_t,f)$ in this case can be bounded by
\begin{align} \label{eq:case-2-a}
  \sum_{f \in P} (1 - p_{\min})(\lambda_{\max}^{\abs{C_{u, \overline{f},\overline{e}}}} + \lambda_{\max}^{\abs{C_{v, \overline{f},\overline{e}}}})
  &\overset{(\star)}{\leq} (1 - p_{\min})\sum_{i = 1}^{\ell} (\lambda_{\max}^{i} + \lambda_{\max}^{\ell + 1 - i}) 
  \leq \frac{2K}{1 - \lambda_{\max}},
\end{align}
where $(\star)$ follows from the fact that $C_{u, \overline{f},\overline{e}}$ and $C_{v, \overline{f},\overline{e}}$ contains at least $i$ and $\ell + 1 - i$ vertices respectively, where $f = (x_{i-1},x_{i})$.

\paragraph{\textbf{Case (2b)}}
As $f \not\in X_t$, it holds that $X_t = S_{\overline{f}, \overline{e}}$, and only edges in $E(C_{u, \overline{f},\overline{e}}, C_{v, \overline{f},\overline{e}})$ may satisfy the requirement of Case (2) (i.e. $f$ connects two components), where we recall that $E(C_{u, \overline{f},\overline{e}}, C_{v, \overline{f},\overline{e}})$ denotes the set of edges between $C_{u, \overline{f},\overline{e}}$ and $C_{v, \overline{f},\overline{e}}$.
% Now, recall that we have $X_t \in \+A(\alpha)$, since $C_u = C_u(X_t) \neq C_v(X_t) = C_v$, the contribution of this case could be bounded as
Therefore, the sum of $d_{X_t,Y_t,f}$ in this case can be bounded by
\begin{align}
  &\nonumber \abs{E(C_{u, \overline{f},\overline{e}}, C_{v, \overline{f},\overline{e}})} \cdot (1 - p_{\min})
  (\lambda_{\max}^{\abs{C_{u, \overline{f},\overline{e}}}} + \lambda_{\max}^{\abs{C_{v, \overline{f},\overline{e}}}})\\ 
  &\nonumber \overset{(\star)}{\leq} \log n \cdot \min \left\{\abs{C_{u, \overline{f},\overline{e}}}, \abs{C_{v, \overline{f},\overline{e}}} \right\} \cdot (1 - p_{\min}) (\lambda_{\max}^{\abs{C_{u, \overline{f},\overline{e}}}} + \lambda_{\max}^{\abs{C_{v, \overline{f},\overline{e}}}}) \\
   \nonumber &\leq 2 \log n \cdot (1 - p_{\min}) \cdot \max_{z > 0} z\; \lambda_{\max}^z \\
   \label{eq:case-2-b} &\leq  \frac{2K}{1-\lambda_{\max}} ,
\end{align}
where $(\star)$ follows from fact that $X_t \in \+G$.

\subsubsection{Case (3)}
% Recall that in this case, we have $C_u(S_{\overline{f}, \overline{e}}) \neq C_{v, \overline{f},\overline{e}}$ and $C_u(S_{f, \overline{e}}) \neq C_v(S_{f, \overline{e}})$.
% Note that if $C_u = C_u(S_{\overline{f}, \overline{e}}) = C_u(S_{f, \overline{e}})$ and $C_v = C_{v, \overline{f},\overline{e}} = C_v(S_{f, \overline{e}})$ happens simultaneously, then
% \begin{align*}
% \frac{\prod_{A \in \kappa(S_{f, e})} (1 + \lambda^A)}{\prod_{B \in \kappa(S_{\overline{f}, e})}(1 + \lambda^B)}
%   &= 
% \frac{\prod_{A \in \kappa(S_{f, e}) \setminus \{C_u\uplus C_v\}} (1 + \lambda^A)}{\prod_{B \in \kappa(S_{\overline{f}, e}) \setminus \{C_u \uplus C_v\}}(1 + \lambda^B)} \\
%   &\overset{(\star)}{=}
%   \frac{\prod_{A \in \kappa(S_{f, \overline{e}})\setminus \{C_u, C_v\}} (1 + \lambda^A)}{\prod_{B \in \kappa(S_{\overline{f}, \overline{e}}) \setminus \{C_u, C_v\}}(1 + \lambda^B)}
%   =
%   \frac{\prod_{A \in \kappa(S_{f, \overline{e}})} (1 + \lambda^A)}{\prod_{B \in \kappa(S_{\overline{f}, \overline{e}})}(1 + \lambda^B)},
% \end{align*}
% where $(\star)$ holds by the fact that $e$ can only effect $C_u, C_v$ and has no intersection with other connected components. This implies $\nu(f \mid e, \sigma) = \nu(f \mid \overline{e}, \sigma)$, which means we should focus on the cases where $C_u(S_{\overline{f}, \overline{e}}) \neq C_u(S_{f, \overline{e}})$ or $C_{v, \overline{f},\overline{e}} \neq C_v(S_{f, \overline{e}})$.
Note that $d(X_t,Y_t,f) = 0$ when 
\begin{align}\label{eq:component}
  C_{u, \overline{f},\overline{e}} = C_{u,f,\overline{e}}\text{ and } C_{v, \overline{f},\overline{e}} =C_{v,f,\overline{e}},
\end{align}
i.e. both $x$ and $y$ are not in $C_{u, \overline{f},\overline{e}} \cup C_{v, \overline{f},\overline{e}}$ where $f=(x,y)$. 
Furthermore, two constraints in~\eqref{eq:component} cannot be violated at the same time, as $f=(x,y)$ does not connect $C_{u, \overline{f},\overline{e}}$ and $C_{v, \overline{f},\overline{e}}$. 
Therefore, it suffices to consider the case where $C_{u, \overline{f},\overline{e}} \neq C_{u,f,\overline{e}}$ and  $C_{v, \overline{f},\overline{e}} = C_{v,f,\overline{e}}$.

% We will analysis the case $C_u(S_{\overline{f}, \overline{e}}) \neq C_u(S_{f, \overline{e}})$ and the other case could be handled by the same analysis.
% Since $C_u(S_{\overline{f}, \overline{e}}) \neq C_u(S_{f, \overline{e}})$, it holds that the deletion of $f$ will break $C_u(S_{f, \overline{e}})$ into $C_u$ and another connected components $C_w$, where $w$ is an endpoint of $f$.
% Hence, we have
% \begin{align*}
%   \nu(f \mid e, \sigma) &= \frac{p_f (1 + \lambda^{C_w \uplus C_u \uplus C_v})}{p_f (1 + \lambda^{C_w \uplus C_u \uplus C_v}) + (1 - p_f)(1 + \lambda^{C_w})(1 + \lambda^{C_u \uplus C_v})} \\
%   \nu(f \mid \overline{e}, \sigma) &= \frac{p_f (1 + \lambda^{C_w \uplus C_u})}{p_f (1 + \lambda^{C_w \uplus C_u}) + (1 - p_f)(1 + \lambda^{C_w})(1 + \lambda^{C_u})}.
% \end{align*}
% By some calculation, we know that
Let $A=C_{u, \overline{f},\overline{e}}, B=C_{v, \overline{f},\overline{e}}, C=C_{u,f,\overline{e}} \setminus A$. A similar calculation yields
\begin{align}\label{eq:pc-bd-case3}
  \nonumber \d(X_t,Y_t,f) &= \abs{ \frac{p_f \tp{1+\*\lambda^{A \cup B}}}{1+\*\lambda^{A \cup B} + (1-p_f) \tp{\*\lambda^{A} + \*\lambda^B}} - \frac{p_f\tp{1+\*\lambda^{A \cup B \cup C}}}{1+\*\lambda^{A \cup B \cup C} + (1-p_f) \tp{\*\lambda^{A \cup C} + \*\lambda^B}}}\\
  &\le (1-p_{\min}) \lambda_{\max}^{\abs{A}}.
\end{align}

We further consider two sub-cases: (a) $f \in X_t$; (b) $f \not\in X_t$.

\paragraph{\textbf{Case (3a)}}
When $f \in X_t$, in order to make contribution, the deletion of $f$ split $C_{u,f, \overline{e}}$ into $A = C_{u,\overline{f}, \overline{e}}$ and $C = C_{u,f, \overline{e}} \setminus A$.
In this case, $f$ must be a bridge in graph $C_{u,f,\overline{e}}$.
Pick an arbitrary spanning tree $T$ rooted at $u$ in $C_{u,f,\overline{e}}$.
Obviously, $f$ must be an edge in $T$.
Let $\ell$ be the size of $T$ and $\{f_1,f_2,\ldots,f_{\ell}\}$ be the edges in $T$ sorted by the size of $C_{u, \overline{f_i},\overline{e}}$ in decreasing order.
Note that $f_j$ must be in component $C_{u,\overline{f_i},\overline{e}}$ for all $1 \le i < j \le \ell$, since $|C_{u, \overline{f_i}, \overline{e}}| \geq |C_{u, \overline{f_j}, \overline{e}}|$ means $f_i$ cannot be an ancestor of $f_j$.
Hence, $|C_{u,\overline{f_i},\overline{e}}| \ge \ell - i + 1$. Together with~\eqref{eq:pc-bd-case3}, we bound the sum of $d(X_t,Y_t,f)$ in this case by
\begin{align} \label{eq:case-3-a}
  % \sum_{i=1}^{\ell-1} (1 - p_{\min}) \lambda_{\max}^{\abs{C_u(S_{\overline{f_i},\overline{e}})}}
  % &= 
  \sum_{i=1}^\ell (1 - p_{\min}) \lambda_{\max}^i \leq \frac{\lambda_{\max}(1 - p_{\min})}{1 - \lambda_{\max}} \le \frac{K}{1-\lambda_{\max}}.
\end{align}

\paragraph{\textbf{Case (3b)}}
When $f \not\in X_t$, $f$ must connects $C_{u, \overline{f},\overline{e}}$ and other component in $(V,S_{\overline{f},\overline{e}})$ except $C_{v, \overline{f},\overline{e}}$.
% Note that this also means $C_u = C_u(X_t)$ and $C_v = C_v(X_t)$.
% Since we have $X_t \in \+A(\alpha)$, it holds that $\abs{E_\Lambda(C_u, C_w)} \leq \alpha \log n \min\{\abs{C_u}, n - \abs{C_u}, \abs{C_w}, n - \abs{C_w} \}$, and contribution of this case could be bounded by a similar manner as \eqref{eq:case-2-b}:
Therefore, the sum of $d(X_t,Y_t,f)$ in this case is bounded by (similar calculation as in \eqref{eq:case-2-b})
\begin{align}
 \abs{E(C_{u, \overline{f},\overline{e}}, V\setminus C_{u, \overline{f},\overline{e}})} \cdot (1 - p_{\min}) \lambda_{\max}^{\abs{C_u}}
  % &\leq \alpha \log n \cdot \min\{\abs{C_u}, \abs{C_w}\} \cdot (1 - p_{\min}) \lambda_{\max}^{\abs{C_u}} \\
   \label{eq:case-3-b} \leq  \frac{K}{1-\lambda_{\max}},
\end{align}
where we use \eqref{eq:pc-bd-case3} and the fact that $S_{\overline{f}, \overline{e}} = X_t \in \+G$.

\subsubsection{Wrapping up}
Recall that by our assumption of $K$ in \eqref{lem:pc-good}, it holds that $K \leq (1 - \lambda_{\max})/27$.
Together with previous analysis, we bound $\sum_{f \in E \setminus \{e\}} d(X_t,Y_t,f)$ by
\begin{align} % parallel ising api
  \sum_{f \in E \setminus \{e\}} d(X_t,Y_t,f) \leq 
  \eqref{eq:case-2-a} + \eqref{eq:case-2-b} + 2 \times \eqref{eq:case-3-a} + 2 \times \eqref{eq:case-3-b} 
  \le \frac{8K}{1-\lambda_{\max}}
  % &\leq 4 \tp{\frac{\lambda_{\max}}{1 - \lambda_{\max}} \frac{K}{\log n} + \alpha K \cdot \frac{\lambda_{\max}^{1/\log(1/\lambda_{\max})}}{\log (1/\lambda_{\max})}} \\
  \leq \frac{1}{2}.
\end{align}
Therefore, it holds that
\begin{align*}
  \E{\abs{X_{t+1} \oplus Y_{t+1}} \mid X_t, Y_t} \leq 1 - \frac{1}{m} \tp{1 - \frac{1}{2}} = 1 - \frac{1}{2m},
\end{align*}
which concludes the proof of \Cref{lem:pc-good}.

\subsection{Bad event happens with small probability (proof of \Cref{lem:bad-event-bound})}\label{sec:bad-event-bound}
Without loss of generality, we will only bound the probability $\Pr{X_t \not\in \+G}$. For simplicity of notation, we denote \emph{update sequence} $\+L_t=(e_1,e_2,\ldots,e_t)$ the chosen edges of the Glauber dynamics in the first $t$ rounds, and denote $\+S(\+L_t) = \{e_1,e_2,\ldots,e_t\}$ be the set of edges of an update sequence $\+L_t$.

Fix $t \ge \theta m \log m$ for some $\theta > 0$.
We claim that $S(\+L_t) = E$ happens with high probability. Furthermore, condition on any updating sequence $\+L_t$ with $S(\+L_t) = E$, the distribution $X_t$ stochastically dominates the product distribution $\nu = \bigotimes_{e \in E} \mathrm{Ber}(1-3K)$, where $K = (1-p_{\min}) \log n$.

\begin{lemma}\label{lm:high-prob}
  If $t \ge \theta m \log m$ for some $\theta > 0$, event $S(\+L_t) \neq E$ occurs with probability at most $m^{1-\theta}$.
\end{lemma}

% API for parallel Ising
\begin{lemma}\label{lm:marginal-bound}
  For every $S \subseteq E$ and $e \in E$, it holds that
  \begin{align*}
    \frac{\mu(S\cup\set{e})}{\mu(S\cup\set{e}) + \mu(S\setminus\set{e})} \geq 1 - 3K.
  \end{align*}
\end{lemma}

\begin{lemma}\label{lm:st-dominate-2}
  For any updating sequence $\+L_t$ satisfying $S(\+L_t)=E$
  the distribution $X_t$ condition on $\+L_t$ stochastically dominates $\nu$, i.e., there exists a coupling $C=(S,T)$ of $(X_t\mid \+L_t)$ and $\nu$ satisfying $T \subseteq S$.
\end{lemma}
We first prove \Cref{lem:bad-event-bound} with \Cref{lm:high-prob} and \Cref{lm:st-dominate-2}.

\begin{proof}[Proof of \Cref{lem:bad-event-bound}]
For any $t \ge \theta m \log m$, it holds that
% \begin{align*}
%   \Pr{X_t \not\in \+G} &= \sum_{(e_1,e_2,\ldots,e_t) \in E^t} \Pr{X_t \not\in \+G \mid e_1,e_2,\ldots,e_t} \Pr{e_1,e_2,\ldots,e_t}\\
%   &= \sum_{\substack{(e_1,e_2,\ldots,e_t) \in E^t\\ \{e_1,e_2,\ldots,e_t\} =E }} \Pr{X_t \not\in \+G \mid e_1,e_2,\ldots,e_t} \Pr{e_1,e_2,\ldots,e_t}\\
%   &+\sum_{\substack{(e_1,e_2,\ldots,e_t) \in E^t\\ \{e_1,e_2,\ldots,e_t\} \neq E }} \Pr{X_t \not\in \+G \mid e_1,e_2,\ldots,e_t} \Pr{e_1,e_2,\ldots,e_t}\\
%   &\le \max_{\substack{(e_1,e_2,\ldots,e_t) \in E^t\\ \{e_1,e_2,\ldots,e_t\} =E }} \Pr{X_t \not\in \+G \mid e_1,e_2,\ldots,e_t}+\Pr{\{e_1,e_2,\ldots,e_t\} \neq E}\\
%   &\le \max_{\substack{(e_1,e_2,\ldots,e_t) \in E^t\\ \{e_1,e_2,\ldots,e_t\} =E }} \Pr{X_t \not\in \+G \mid e_1,e_2,\ldots,e_t}+ m^{-4}.
%   % \le \sum_{S \in \+C} (3K)^{\abs{E(S,V \setminus S)}} 
%   % % &\text{(By \Cref{lm:st-dominate})}\\ 
%   % \le \sum_{S \in \+C} (3K)^{\abs{S} \log n} 
%   % % & \text{(definition of $\+C$)}\\
%   % \le \sum_{j=1}^{+\infty} n^j n^{j \log(3k)}
%   % \le n^{\log (27K)},
% \end{align*}
\begin{align*}
  \Pr{X_t \not\in \+G} 
  % &= \sum_{\+L_t \in E^t} \Pr{X_t \not\in \+G \mid \+L_t} \Pr{\+L_t}
  &= \sum_{\substack{\+L_t \in E^t\\ S(\+L_t)=E }} \Pr{X_t \not\in \+G \mid \+L_t} \Pr{\+L_t}
  +\sum_{\substack{\+L_t \in E^t\\ S(\+L_t) \neq E }} \Pr{X_t \not\in \+G \mid \+L_t} \Pr{\+L_t}\\
  &\le \max_{\substack{\+L_t \in E^t\\ S(\+L_t) =E }} \Pr{X_t \not\in \+G \mid \+L_t}+\Pr{S(\+L_t) \neq E}
  % \le \max_{\substack{\+L_t \in E^t\\ S(\+L_t) =E }} \Pr{X_t \not\in \+G \mid \+L_t}+ m^{-4}.
  % \le \sum_{S \in \+C} (3K)^{\abs{E(S,V \setminus S)}} 
  % % &\text{(By \Cref{lm:st-dominate})}\\ 
  % \le \sum_{S \in \+C} (3K)^{\abs{S} \log n} 
  % % & \text{(definition of $\+C$)}\\
  % \le \sum_{j=1}^{+\infty} n^j n^{j \log(3k)}
  % \le n^{\log (27K)},
\end{align*}

% By \Cref{lm:high-prob}, the second term is upper bounded by $m^{-4}$.
Fix an update sequence $\+L_t$ satisfying $S(\+L_t)=E$. 
By \Cref{lm:st-dominate-2}, the definition of good event $\+G$ and the definition of $\nu$,
\begin{align*}
  \Pr{X_t \not\in \+G \mid \+L_t}
  \le \sum_{S \in \+C} (3K)^{\abs{E(S,V \setminus S)}} 
  % &\text{(By \Cref{lm:st-dominate})}\\ 
  \le \sum_{S \in \+C} (3K)^{\abs{S} \log n} 
  % & \text{(definition of $\+C$)}\\
  \le \sum_{j=1}^{+\infty} n^j n^{j \log(3k)}
  \le n^{\log (27K)}.
\end{align*}

Furthermore, by \Cref{lm:high-prob}, the probability of $\+S(\+L_t) \neq E$ is upper bounded by $m^{1-\theta}$. 
Therefore,
\begin{align*}
    \Pr{X_t \not\in \+G} &\le n^{\log (27 K)} + m^{1-\theta}. \qedhere 
\end{align*}
\end{proof}

In the rest of this section, we are dedicated to proving \Cref{lm:high-prob} and \Cref{lm:st-dominate-2}.

\begin{proof}[Proof of \Cref{lm:high-prob}]
  Fix $t \ge \theta m \log m$.
  For each edge $e \in E$, the probability that $e \not\in \+L_t$ is at most
  \begin{align*}
    \Pr{e \not \in \+L_t} \le \tp{1-\frac{1}{m}}^t \le \tp{1-\frac{1}{m}}^{\theta m \log m} \leq m^{-\theta}. 
  \end{align*}
  By union bound, 
  \begin{align*}
    \Pr{S(\+L_t) \neq E} &\le \sum_{e \in E} \Pr{e \not \in \+L_t} \le m^{1-\theta}. \qedhere
  \end{align*}
\end{proof}

\begin{proof}[Proof of \Cref{lm:st-dominate-2} assuming \cref{lm:marginal-bound}]
  Instead of proving the original statement, we will prove that for any update sequence $\+L_t \in E$, there exists a coupling $(S,T)$ of $(X_t \mid \+L_t)$ and $\nu$ satisfying that $(T \cap R) \subseteq (S \cap R)$, where $R = S(\+L_t)$.
  We remark that this is stronger compared to the original statement. 

  We prove by induction on $t$.
  Base case $t=0$ follows from $R = \emptyset$.
  Now we assume our assumption holds for any $t' \le t$.
  Fix an update sequence $\+L_{t+1} = (e_1,e_2,\ldots,e_{t+1})$. Let $\+L_t=(e_1,e_2,\ldots,e_t)$ be the update sequence. By induction hypothesis, there exists a coupling $(S,T)$ of $(X_t \mid \+L_t)$ and $\nu$ satisfying $(T \cap R) \subseteq (S \cap R)$, where $R = S(\+L_t)$. 
  We design the following coupling $(S',T')$ of $(X_t \mid \+L_{t+1})$ and $\nu$ based on $(S, T)$:
  \begin{itemize}
    % \item sample $(S,T)$ from coupling $\+C$;
    \item let $q_0 = \frac{\mu(S \cup \{e_{t+1}\} )}{\mu(S \cup \{e_{t+1}\}) + \mu(S \setminus \{e_{t+1}\})}$ and $q_1 = 1-3K$;
    \item draw $r \sim \mathrm{Uniform}(0,1)$;
    \begin{itemize}
      \item if $r < q_0$, set $S' = S \cup \{e_{t+1}\}$; Otherwise, set $S' = S \setminus \{e_{t+1}\}$;
      \item if $r < q_1$, set $T' = T \cup \{e_{t+1}\}$; Otherwise, set $T' = T \setminus \{e_{t+1}\}$.
    \end{itemize}
  \end{itemize} 
  Note that $q_0$ is the probability of adding $e_{t+1}$ to $S$ in the update step of Glauber dynamics.
  It can be verified that $(S', T')$ is indeed a coupling of $(X_{t+1} \mid \+L_{t+1})$ and $\nu$. 
  Moreover, if $q_0 \ge q_1$, it holds that $(T' \cap R') \subseteq (S' \cap R')$, where $R' = \+S(\+L_{t+1})$. 
  Hence, it remains to show $q_0 \ge q_1 = 1-3K$, which follows from \cref{lm:marginal-bound}.
  %Let $S \subseteq E$ be the subsets generated in the first step of coupling procedure, and $e=(u,v)$ be the edge picked from the second step. 
  This concludes the proof of \Cref{lm:st-dominate-2}.
\end{proof}

\begin{proof}[Proof of \cref{lm:marginal-bound}]
  For convenience, let $e = (u, v)$.
  We consider the following cases.
  \begin{enumerate}
    \item $u, v$ are in the same connected component in graph $(V, S \setminus \{e\})$.
    In this case,
    \begin{align*}
      \frac{\mu(S \cup \{e\} )}{\mu(S \cup \{e\}) + \mu(S \setminus \{e\})} = p_{e} \ge 1 - K/\log n \geq 1-3K.
    \end{align*}
    \item $u, v$ are in different connected components in graph $(V, S \setminus \{e\})$. 
    Let $C_u$ and $C_v$ be the connected components that $u$ and $v$ are in respectively.
    In this case,
    \begin{align*}
      \frac{\mu(S \cup \{e\} )}{\mu(S \cup \{e\}) + \mu(S \setminus \{e\})} = \frac{p_e (1+\*\lambda^{C_u\cup C_v})}{1+ \*\lambda^{C_u \cup C_v} + (1-p_e)\tp{\*\lambda^{C_u} + \*\lambda^{C_v}}} \ge \frac{p_e}{3-2p_e}.
    \end{align*}
    Recall that $p_e \ge 1-\frac{K}{\log n} \ge 1-K$. Therefore, $\frac{\mu(S \cup \{e\} )}{\mu(S \cup \{e\}) + \mu(S \setminus \{e\})} \ge \frac{1-K}{1+2K} \ge 1-3K$. \qedhere
  \end{enumerate}
\end{proof}

\section{Acknowledgement}
We would like to thank Jingcheng Liu and Yitong Yin for inspiring discussions and invaluable comments on manuscripts of this paper.
We would also like to thank Weiming Feng for helpful discussions on manuscripts.
X.Z. would like to thank Chunyang Wang for pointing out a dynamic connectivity structure.

\bibliographystyle{alpha}
\bibliography{refs}

\appendix

\section{Missing proofs}

\subsection{Proof of \Cref{thm:main}}\label{sec:main-proof}
Let $\+A(G=(V,E),\*p,\*\lambda,\epsilon)$ be the approximate sampler in \Cref{thm:alg-RC} for distribution of random cluster model specified by graph $G$ and parameters $\*p$ and $\*\lambda$ within total variation distance $\epsilon$.
The sampler $\+B(G=(V,E),\*\beta,\*\lambda,\epsilon)$ for Gibbs distribution of Ising model is defined as follows.
\begin{itemize}
  \item initialize $X \gets \emptyset$ and let $\*p = \tp{1 - \beta_i^{-1}}_{i \in V}$;
  \item draw $S$ from $\+A(G,\*p,\*\lambda,\epsilon)$;
  \item for each $C \in \kappa(V,X)$, add $C$ to $X$ with probability $\frac{\*\lambda^C}{1+\*\lambda^C}$;
  \item return $X$.
\end{itemize}
By a standard coupling argument and \Cref{prop:ES-coupling}, $X$ drawn in $\+B(G,\*\beta,\*\lambda,\epsilon)$ satisfies
\begin{align*}
  d_{\mathrm{TV}}(X,\mu^{\mathrm{Ising}}_{\*\beta,\*\lambda}) \le \epsilon.
\end{align*}
This concludes the proof of \Cref{thm:main}.

\subsection{Proof of \Cref{lem:mixing-field}}\label{sec:field-proof}
% \todo{remove \eqref{eq:mixing-decay}}
% For convenience, we denote $\mu^{\-{RC}}_{E, \*p, \*\lambda}$ as $\mu$.
% First, we claim that $\mu(E) \ge \tp{p_{\min}/2}^{n^2}$.
By \Cref{lem:SI-RC}, \Cref{lem:field-dynamics-mixing}, and~\eqref{eq:mixing-decay}, it suffices to show
\begin{align*}
  \mu^{\mathrm{RC}}_{E,\*p,\*\lambda}(E) \ge \tp{\frac{p_{\min}}{2}}^{n^2}.
\end{align*}
% \begin{align*}
%   \DTV{X_T}{\mu}
%   &\leq \sqrt{2^{-1} (1 - \kappa)^T D_{\-{KL}}({\*1}_E \parallel \mu)}
%     \leq \sqrt{2^{-1} (1 - \kappa)^T \log(1/\mu(E))} \leq \epsilon,
% \end{align*}
% where the last inequality follows from our choice of $T$.

By \Cref{prop:partition-func}, $Z^{\-{RC}}_{E, \*p, \*\lambda} \cdot \*\beta^E = Z^{\-{Ising}}_{\*\beta, \*\lambda} \leq \*\beta^E \cdot 2^n$, which implies $Z^{\-{RC}}_{E, \*p, \*\lambda} \leq 2^{n}.$
Therefore,
\begin{align*}
  \mu^{\mathrm{RC}}_{E,\*p,\*\lambda}(E) \ge \frac{\*p^E\tp{1+\*\lambda^V}}{2^n} \ge \tp{\frac{p_{\min}}{2}}^{2^n}.
\end{align*}
%For any $S \subseteq E$,
%\begin{align*}
%  \mu(S) \geq \tp{\max_{X \subseteq E, i \in V} \mu\tp{X \mid \+P_{X,V \setminus \{i\}}}}^{m}.
%\end{align*}
%When $i \in X$, it holds by~\eqref{eq:rc-transition} that
%\begin{align*}
%  p_{\max} \ge \mu\tp{X \mid \+P_{X,V \setminus \{i\}}} \ge \min_{\substack{e \in E, C_1,C_2 \in V\\C_1 \cap C_2 = \emptyset}} \left\{ \frac{p_e(1+\*\lambda^{C_1 \cup C_2})}{1+\*\lambda^{C_1 \cup C_2} + (1-p_e)(\*\lambda^{C_1} + \*\lambda^{C_2})} \right\} \ge \frac{p_{\min}}{2-p_{\min}} \ge p_{\min}/2.  
%\end{align*}
%Therefore, $\mu\tp{X \mid \+P_{X,V \setminus \{i\}}} \ge \min \{1-p_{\max},p_{\min}/2\} = \delta$.

\subsection{Proof of \Cref{lem:field-dynamics-mixing}} \label{append:field}
We give a proof of \Cref{lem:field-dynamics-mixing} in this section for completeness.
Recall that in \Cref{sec:field-dynamics}, we consider a distribution $\overline{\mu}$ over $2^U$ on the ground set $U$.
In each round, the field dynamics $\overline{P^{\-{FD}}_\theta}$ for $\overline{\mu}$ with parameter $\theta \in (0, 1)$ updates a configuration $X \in 2^U$ as
\begin{itemize}
\item sample $S' \sim \bigotimes_{u \in U} \-{Ber}(\theta)$ and let $S = S' \cup X$;
\item update $X$ according to distribution $(\theta^{-1} * \overline{\mu})( \cdot \mid \+P_{X, U\setminus S})$.
\end{itemize}
We want to prove that when $\*\lambda * \overline{\mu}$ is $C$-spectrally independent under all pinnings for all $\*\lambda \in \mathbb{R}^U_{>0}$, then for $\theta \in (0, 1)$, and any distribution $\overline{\nu}$ that is absolutely continous with respect to $\overline{\mu}$, we have
\begin{align} \label{eq:ent-FD-target-raw}
  D_{\-{KL}}(\overline{\nu}\; \overline{P^{\-{FD}}_\theta} \parallel \overline{\mu}\; \overline{P^{\-{FD}}_\theta}) &\leq (1 - \kappa) D_{\-{KL}}(\overline{\nu} \parallel \overline{\mu}),
\end{align}
where $\kappa$ is defined as $\kappa = (\theta / \e)^{C + 3}$.

Note that the field dynamics we use here is not in its standard form in previous works~\cite{chen2021rapid}.
In order to be more compatible with previous results, we first transform the setting we use to a more standard version. 
Specifically, let $\mu$ and $\nu$ be the distribution over $2^U$ defined as
\begin{align*}
  \forall S\subseteq E, \quad \mu(S) := \overline{\mu}(E\setminus S), \text{ and } \nu(S) := \overline{\nu}(E\setminus S).
\end{align*}
It is standard to check that $\nu$ is absolutely continous with respect to $\mu$ and $\*\lambda * \mu$ is $C$-spectrally independent under all pinnings for all $\*\lambda \in \mathbb{R}^U_{>0}$.
In each round, the field dynamics $P^{\-{FD}}_\theta$ with parameter $\theta \in (0, 1)$ falls into its standard form that updates a configuration $X \in 2^U$ as
\begin{itemize}
\item sample $S' \sim \bigotimes_{u \in U} \-{Ber}(\theta)$ and let $S = S' \cup (E\setminus X)$;
\item update $X$ according to the distribution $(\theta * \mu)(\cdot \mid \+P_{X, U\setminus S})$.
\end{itemize}
Note that $P^{\-{FD}}_\theta$ is exact $\overline{P^{\-{FD}}_\theta}$ except we exchange the role of ``in'' and ``out'' for each element in $U$, and it could be easily checked that 
\begin{align*}
  \forall X, Y \subseteq E, \quad P^{\-{FD}}_\theta(X, Y) = \overline{P^{\-{FD}}_\theta}(E\setminus X, E\setminus Y),
\end{align*}
which implies that
\begin{align*}
  \forall X \subseteq E, \quad \mu P^{\-{FD}}(X) = \overline{\mu}\; \overline{P^{\-{FD}}_\theta} (E\setminus X) \text{ and } \nu P^{\-{FD}}(X) = \overline{\nu}\; \overline{P^{\-{FD}}_\theta} (E\setminus X).
\end{align*}
Hence, in order to prove \Cref{eq:ent-FD-target-raw}, it is sufficient for us to prove
\begin{align} \label{eq:ent-FD-target}
  D_{\-{KL}}(\nu P^{\-{FD}}_\theta \parallel \mu P^{\-{FD}}_\theta) &\leq (1 - \kappa) D_{\-{KL}}(\nu \parallel \mu),
\end{align}
for $\kappa = (\theta / \e)^{C + 3}$, under the assumption that $\*\lambda * \mu$ is $C$-spectrally independent under all pinnings for all $\*\lambda \in \mathbb{R}^U_{> 0}$ and $\nu$ is absolutely continuous with respect to $\mu$.

For convenience, we just inherit the notation that is used in \cite{chen2021optimalIsing}.
Suppose $\mu$ is a distribution over $2^U$ over the ground set $U$, it could also be interpreted as a distribution over $\{\0, \1\}^U$.
Let $\theta \in (0, 1)$ and $\pi := \theta * \mu$.
Formally, for every $\sigma \in \{\0, \1\}^U$, we have
\begin{align*}
  \pi(\sigma) := \frac{\mu(\sigma) \theta^{\norm{\sigma}_+}}{Z_\pi},
\end{align*}
where $Z_\pi := \sum_{\sigma \in \{\0, \1\}^U} \mu(\sigma) \theta^{\norm{\sigma}_+}$ and $\norm{\sigma}_+$ denotes the number of $\1$ in the vector $\sigma$.
For a function $f:\{\0, \1\}^U \to \mathbb{R}_{\geq 0}$ and a distribution $\mu$ over $\{\0, \1\}^U$, the entropy is defined as $\Ent[\mu]{f} := \E[\mu]{f \log f} - \E[\mu]{f} \log \E[\mu]{f}$.

\begin{lemma}[\text{\cite[Lemma 2.3]{chen2021optimalIsing}}] \label{lem:key} % field block factorization
Let $\theta \in (0, 1)$ be a real number.
Let $\mu$ be a distribution over $\{\0,\1\}^U$ and $\eta > 0$.
If  ${\mu}$ is $C$-spectrally independent under all pinnings for all $\*\lambda \in \mathbb{R}_{>0}^V$, then $\mu$ satisfies the following inequality for any function $f: \{\0, \1\}^U \to \mathbb{R}_{\geq 0}$ with $\kappa = \tp{\frac{\theta}{\e}}^{C + 3}$,
\begin{align}  \label{eq:mag-bf-ent}
  \Ent[\mu]{f} \leq \kappa^{-1} \cdot \frac{Z_\pi}{\theta^{\abs{U}}} \sum_{R \subseteq U} (1 - \theta)^{\abs{R}} \theta^{\abs{U} - \abs{R}} \cdot \pi_R(\*{1}_R) \cdot \Ent[\pi^{{\*1}_R}]{f},
\end{align}
where $\*{1}_R$ is the all-$1$ vector on $R$.
\end{lemma}

\begin{remark}
  In \cite{chen2021optimalIsing}, \eqref{eq:mag-bf-ent} is called $\theta$-magnetized block factorization of entropy.
  % Here, to avoid introducing too much concepts, we use the expanded version of it instead of the original version.
\end{remark}

Note that the field dynamics could be decomposed into two components that mimics the so-called down-up walk.
Let $\Omega := \{\sigma \in \{\0, \1\}^U\}$ and $\omega := \{\*{1}_R \mid R\subseteq U\}$
that is
\begin{align*}
  P^{\-{FD}}_\theta = P^{\downarrow}P^{\uparrow},
\end{align*}
where $P^{\downarrow} \in \mathbb{R}_{\geq 0}^{\Omega \times \omega}$ and $P^{\uparrow} \in \mathbb{R}_{\geq 0}^{\omega \times \Omega}$ are defined as follow: $\forall \sigma \in \Omega, {\*1}_R \in \omega$ we have
\begin{align*}
  P^{\downarrow}(\sigma, \*{1}_R) &:= \*1[R \subseteq \sigma^{-1}(\1)] \tp{1 - \theta}^{\abs{R}} \theta^{\norm{\theta}_+ - \abs{R}} \\
  \text{and} \quad 
  P^{\uparrow}(\*{1}_R, \sigma) &:= \*1[R \subseteq \sigma^{-1}(\1)] \pi^{\*{1}_R}(\sigma),
\end{align*}
where $\sigma^{-1}(\1) := \{u \in U \mid \sigma_u = \1\}$ denotes the set of $\1$-spin elements in $U$ according to $\sigma$.
Moreover, let $\mu_0 := \mu P^{\downarrow}$, these two operator have the following adjoint property: $\forall \sigma \in \Omega, \*{1}_R \in \omega$, 
\begin{align*}
  \mu(\sigma) P^{\downarrow}(\sigma, \*{1}_R) = \mu_0(\*{1}_R) P^{\uparrow}(\*{1}_R, \sigma).
\end{align*}
Without loss of generality, if we assume $R \subseteq \sigma^{-1}(\1)$, then we have
\begin{align*}
  \mu_0(\*{1}_R) P^{\uparrow}(\*{1}_R, \sigma)
  &= \sum_{\tau: \tau_R = \*{1}_R} (1 - \theta)^{\abs{R}} \theta^{\norm{\tau}_+ - \abs{R}} \mu(\tau) \cdot \pi^{\*{1}_R}(\sigma) \\
  &= \sum_{\tau: \tau_R = \*{1}_R} (1 - \theta)^{\abs{R}} \theta^{\norm{\tau}_+ - \abs{R}} \cdot \frac{\mu(\tau)\pi(\sigma)}{\pi_R(\*{1}_R)} \\
  &= \sum_{\tau: \tau_R = \*{1}_R} (1 - \theta)^{\abs{R}} \theta^{\norm{\sigma}_+ - \abs{R}} \cdot \frac{\pi(\tau)\mu(\sigma)}{\pi_R(\*{1}_R)} \\
  &= \mu(\sigma) (1 - \theta)^{\abs{R}} \theta^{\norm{\sigma}_+ - \abs{R}} \sum_{\tau: \tau_R = \*{1}_R} \frac{\pi(\tau)}{\pi_R(\*{1}_R)} \\
  &= \mu(\sigma) \cdot (1 - \theta)^{\abs{R}}\theta^{\norm{\sigma}_+ - \abs{R}}
   = \mu(\sigma) P^{\downarrow}(\sigma, \*{1}_R).
\end{align*}

Let $\nu$ be a distribution over $\{\0, \1\}^U$, let $f = \frac{\nu}{\mu}$.
By standard result, adjoint property gives us
\begin{align} \label{eq:adjoint}
  \frac{\nu P^\downarrow}{\mu P^\downarrow} = P^\uparrow f.
\end{align}

Now, we are ready to prove \Cref{lem:field-dynamics-mixing}.
We claim that when $f = \frac{\nu}{\mu}$, \eqref{eq:mag-bf-ent} is equivalent to
\begin{align} \label{eq:KL-decay-target}
  D_{\-{KL}}(\nu \parallel \mu) \leq \kappa^{-1} \tp{D_{\-{KL}}(\nu \parallel \mu) - D_{\-{KL}}(\nu P^{\downarrow} \parallel \mu P^{\downarrow})},
\end{align}
which is equivalent to
\begin{align*}
  D_{\-{KL}}(\nu P^{\downarrow} \parallel \mu P^{\downarrow}) &\leq (1 - \kappa) D_{\-{KL}}(\nu \parallel \mu).
\end{align*}
Then, \eqref{eq:ent-FD-target} (and \Cref{lem:field-dynamics-mixing}) follows directly from the data processing inequality as
\begin{align*}
  D_{\-{KL}}(\nu P^{\-{FD}}_\theta \parallel \mu P^{\-{FD}}_\theta) = D_{\-{KL}} (\nu P^\downarrow P^\uparrow \parallel \mu P^\downarrow P^\uparrow) \leq 
  D_{\-{KL}}(\nu P^{\downarrow} \parallel \mu P^{\downarrow}) \leq (1 - \kappa) D_{\-{KL}}(\nu \parallel \mu).
\end{align*}
Now, we are going to prove \eqref{eq:mag-bf-ent} $\Leftrightarrow$ \eqref{eq:KL-decay-target} by a brute force calculation.
First, note that
\begin{align}
  \frac{Z_\pi}{\theta^{\abs{U}}}
  \nonumber & \sum_{R \subseteq U} (1 - \theta)^{\abs{R}} \theta^{\abs{U} - \abs{R}} \cdot \pi_R(\*{1}_R) \cdot \Ent[\pi^{\*{1}_R}]{f} \\
  \label{eq:tar-1} &= \frac{Z_\pi}{\theta^{\abs{U}}} \sum_{R\subseteq U} (1 - \theta)^{\abs{R}}\theta^{\abs{U}-\abs{R}} \pi_R(\*{1}_R) \E[\pi^{\*{1}_R}]{f \log f} \quad -  \\
  \label{eq:tar-2} & \quad  \frac{Z_\pi}{\theta^{\abs{U}}} \sum_{R\subseteq U} (1 - \theta)^{\abs{R}}\theta^{\abs{U}-\abs{R}} \pi_R(\*{1}_R) \E[\pi^{\*{1}_R}]{f} \log \E[\pi^{\*{1}_R}]{f}.
\end{align}
We will show that
\begin{align*}
\eqref{eq:tar-1} = D_{\-{KL}}(\nu \parallel \mu) = \Ent[\mu]{f}
\quad \text{and} \quad
\eqref{eq:tar-2} = D_{\-{KL}}(\nu P^\downarrow \parallel \mu P^\downarrow).
\end{align*}
We start from \eqref{eq:tar-1} and note that
\begin{align*}
  \eqref{eq:tar-1}
  &= \frac{Z_\pi}{\theta^{\abs{U}}} \sum_{R\subseteq U} (1 - \theta)^{\abs{R}} \theta^{\abs{U} - \abs{R}} \sum_{\sigma: \sigma_R = \*{1}_R} \pi(\sigma) f(\sigma) \log f(\sigma) \\
  &= \frac{Z_\pi}{\theta^{\abs{U}}} \sum_{\sigma \in \Omega} \pi(\sigma) \sum_{R\subseteq \sigma^{-1}(\1)} (1 - \theta)^{\abs{R}} \theta^{\abs{U} - \abs{R}} \cdot f(\sigma) \log f(\sigma) \\
  &= \sum_{\sigma \in \Omega} \mu(\sigma) \sum_{R\subseteq \sigma^{-1}(\1)} (1 - \theta)^{\abs{R}} \theta^{\norm{\sigma}_+ - \abs{R}} \cdot f(\sigma) \log f(\sigma) \\
  &= \sum_{\sigma \in \Omega} \mu(\sigma) \cdot f(\sigma) \log f(\sigma) \sum_{R\subseteq \sigma^{-1}(\1)} (1 - \theta)^{\abs{R}} \theta^{\norm{\sigma}_+ - \abs{R}} \cdot \\
  &= \sum_{\sigma \in \Omega} \mu(\sigma) \cdot f(\sigma) \log f(\sigma) = \Ent[\mu]{f}.
\end{align*}
Now, we only left to prove \eqref{eq:tar-2} $= D_{\-{KL}}(\nu P^\downarrow \parallel \mu P^\downarrow)$.
By the definition of KL-divergence, we have
\begin{align*}
  D_{\-{KL}}(\nu P^\downarrow \parallel \mu P^\downarrow)
  &= \sum_{R\subseteq U} \mu P^\downarrow (\*{1}_R) \cdot \frac{\nu P^\downarrow (\*{1}_R)}{\mu P^\downarrow (\*{1}_R)} \log \frac{\nu P^\downarrow (\*{1}_R)}{\mu P^\downarrow (\*{1}_R)} \\
 (\text{by \eqref{eq:adjoint}}) \quad &= \sum_{R\subseteq U} \mu P^\downarrow (\*{1}_R) \cdot P^\uparrow f(\*{1}_R) \log P^\uparrow f(\*{1}_R) \\
  &\overset{(\star)}{=} \sum_{R\subseteq U} \sum_{\sigma: \sigma_R = \*{1}_R} \mu(\sigma) (1 - \theta)^{\abs{R}} \theta^{\norm{\sigma}_+ - \abs{R}} \cdot \E[\pi^{\*{1}_R}]{f} \log \E[\pi^{\*{1}_R}]{f} \\
  &= \frac{Z_\pi}{\theta^{\abs{U}}} \sum_{R\subseteq U} \sum_{\sigma: \sigma_R = \*{1}_R} \pi(\sigma) (1 - \theta)^{\abs{R}} \theta^{\abs{U} - \abs{R}} \cdot \E[\pi^{\*{1}_R}]{f} \log \E[\pi^{\*{1}_R}]{f} \\
  &= \frac{Z_\pi}{\theta^{\abs{U}}} \sum_{R\subseteq U}  (1 - \theta)^{\abs{R}} \theta^{\abs{U} - \abs{R}} \tp{\sum_{\sigma: \sigma_R = \*{1}_R} \pi(\sigma)} \cdot \E[\pi^{\*{1}_R}]{f} \log \E[\pi^{\*{1}_R}]{f} \\
  &= \frac{Z_\pi}{\theta^{\abs{U}}} \sum_{R\subseteq U}  (1 - \theta)^{\abs{R}} \theta^{\abs{U} - \abs{R}} \pi_R(\*{1}_R) \cdot \E[\pi^{\*{1}_R}]{f} \log \E[\pi^{\*{1}_R}]{f} 
   = \eqref{eq:tar-2},
\end{align*}
where in $(\star)$ we use the fact that
\begin{align*}
  \mu P^\downarrow (\*{1}_R) = \sum_{\sigma: \sigma_R = \*{1}_R} \mu(\sigma) (1 - \theta)^{\abs{R}} \theta^{\norm{\sigma}_+ - \abs{R}}
  \quad \text{and} \quad
  P^\uparrow f(\*{1}_R) = \E[\pi^{\*{1}_R}]{f}.
\end{align*}

\section{Proof for the second part of~\Cref{thm:good-mixing}}\label{append:proof}
Let $(X_t)_{t \ge 0}$ and $(Y_t)_{t \ge 0}$ be Glauber dynamics starting from $E$ and the stationary distribution $\mu$ respectively.
Just as in the proof of the first part, we present the following lemma. The proof of \Cref{lem:bad-event-bound-2} follows a similar fashion to that of \Cref{lem:bad-event-bound}, and we shall omit the proof for brevity.
\begin{lemma}\label{lem:bad-event-bound-2}
  For any $t \ge 0$, it holds that
  \begin{align*}
    \Pr[]{X_t \not\in \+G} \le n^{\log(27K)} \text{ and } \Pr[]{Y_t \not\in \+G} \le n^{\log (27K)}.
  \end{align*}
\end{lemma}

\begin{proof}[Proof for the second part of~\Cref{thm:good-mixing}]
Let $T = 2m(\log m + \log \tp{2/\epsilon})$ and $\+E_t$ be the event that $X_t \not\in \+G$ or $Y_t \not\in \+G$ happens. By \Cref{lem:bad-event-bound-2} and union bound,
 % and $\+E$ be the event 
 % \begin{align*}
 %   \E{\abs{X_{t+1} \oplus Y_{t+1}} \mid X_t, Y_t} \leq \tp{1 - \frac{1}{2m}} \abs{X_t \oplus Y_t}, \quad \forall 0 \le t < T.
 % \end{align*}
  \begin{align*}
    \forall 0\leq t \leq T - 1, \quad \Pr{\overline{\+E_t}} &\leq 2 n^{\log(27K)}.
  \end{align*}
  %\begin{align*}
  %  \Pr{\+E} \ge 1 - 2Tn^{\log (27K)} = 1 - 4m^2 n^{\log (27K)} - 4m\log \tp{\frac{2}{\epsilon}} \ge 1-n^{\log(10^5 K)} - n^{\log(10^5 K)} \log \tp{\frac{2}{\epsilon}}. 
  %\end{align*}
  When $K \leq 10^{-5}\exp\tp{- \frac{\log (8/\epsilon)}{\log n}}$, it holds that
  \begin{align*}
    2n^{\log(27K)} \cdot 2m^2 \leq 4 n^{4 + \log(27K)}  \leq \frac{\epsilon}{2}.
  \end{align*}
  %\begin{align*}
  %  n^{\log (10^5 K)} \log \tp{\frac{2}{\epsilon}} \le \frac{n^{-\frac{2 \log \epsilon}{\log n}}}{4 \epsilon} = \frac{\epsilon}{4}.    
  %\end{align*}
  %Therefore, 
  %\begin{align*}
  %  \Pr{\+E} \ge 1-\frac{\epsilon}{2}.
  %\end{align*}
  Together with \Cref{lem:pc-good} and \Cref{lem:cfs}, we have
  \begin{align*}
    \Pr{X_T \neq Y_T} &\le \tp{1-\frac{1}{2m}}^T + \frac{\epsilon}{2} \leq \epsilon. \qedhere
  \end{align*}
\end{proof}
\end{document}